\documentclass[a4paper,10pt]{amsart}

\usepackage{enumerate}

\usepackage{calligra}
\usepackage[latin1]{inputenc}
\usepackage{amsmath,amsthm,amssymb}
\usepackage{dsfont}
\usepackage[all]{xy}
\usepackage{mathtools}
\usepackage{tikz}
\usetikzlibrary{arrows,calc}

\newcommand{\al}{\alpha}
\newcommand{\eps}{\varepsilon}
\newcommand{\ld}{\lambda}

\newcommand{\fl}{\rightarrow}
\newcommand{\gfl}{\longrightarrow}

\newcommand{\dr}{\ar@{->}[r]}
\newcommand{\dri}{\ar@{>->}[r]}
\newcommand{\drp}{\ar@{-->}[r]}
\newcommand{\dre}{\ar@{->>}[r]}
\newcommand{\dreg}{\ar@{=}[r]}
\newcommand{\drm}{\ar@{^{(}->}[r]}
\newcommand{\ddr}{\ar@{->}[rr]}
\newcommand{\ddre}{\ar@{->>}[rr]}
\newcommand{\ddreg}{\ar@{=}[rr]}
\newcommand{\ha}{\ar@{->}[u]}
\newcommand{\hae}{\ar@{->>}[u]}
\newcommand{\hap}{\ar@{-->}[u]}
\newcommand{\ham}{\ar@{^{(}->}[u]}
\newcommand{\hham}{\ar@{^{(}->}[uu]}
\newcommand{\hag}{\ar@{->}[ul]}
\newcommand{\hagm}{\ar@{^{(}->}[ul]}
\newcommand{\hagp}{\ar@{-->}[ul]}
\newcommand{\hdr}{\ar@{->}[ur]}
\newcommand{\hdrm}{\ar@{^{(}->}[ur]}
\newcommand{\hdri}{\ar@{>->}[ur]}
\newcommand{\hdre}{\ar@{->>}[ur]}
\newcommand{\hdrp}{\ar@{-->}[ur]}
\newcommand{\hddrp}{\ar@{-->}[urr]}
\newcommand{\bas}{\ar@{->}[d]}
\newcommand{\bbas}{\ar@{->}[dd]}
\newcommand{\basm}{\ar@{^{(}->}[d]}
\newcommand{\basi}{\ar@{>->}[d]}
\newcommand{\base}{\ar@{->>}[d]}
\newcommand{\basp}{\ar@{-->}[d]}
\newcommand{\baseg}{\ar@{=}[d]}
\newcommand{\bbaseg}{\ar@{=}[dd]}
\newcommand{\bdr}{\ar@{->}[dr]}
\newcommand{\bdre}{\ar@{->>}[dr]}
\newcommand{\bbdr}{\ar@{->}[ddr]}
\newcommand{\bddr}{\ar@{->}[drr]}
\newcommand{\gau}{\ar@{->}[l]}
\newcommand{\bg}{\ar@{->}[dl]}
\newcommand{\bgm}{\ar@{^{(}->}[dl]}
\newcommand{\bgp}{\ar@{-->}[dl]}
\newcommand{\bggp}{\ar@{-->}[dll]}
\newcommand{\bbgp}{\ar@{-->}[ddl]}
\newcommand{\bdro}{\ar@/_.5pc/@{->}[dr]^0}
\newcommand{\upex}{\ar@/_.7pc/@{-->}[u]}
\newcommand{\hdrex}{\ar@/^.5pc/@{-->}[ur]}
\newcommand{\bgex}{\ar@/^.5pc/@{-->}[dl]}
\newcommand{\gex}{\ar@/^.5pc/@{-->}[l]}

\newcommand{\ac}{\mathcal{A}}

\newcommand{\cat}{\mathcal{C}}
\newcommand{\dc}{\mathcal{D}}

\newcommand{\rc}{\mathcal{R}}
\renewcommand{\sc}{\mathcal{S}}
\newcommand{\tc}{\mathcal{T}}

\newcommand{\wc}{\mathcal{W}}

\newcommand{\shift}{\Sigma}

\newtheorem{theo}{Theorem}[section]
\newtheorem{prop}[theo]{Proposition}
\newtheorem{lemma}[theo]{Lemma}
\newtheorem{cor}[theo]{Corollary}
\newtheorem{defi}[theo]{Definition}
\newtheorem{rk}[theo]{Remark}

\tikzset{htpcof/.style={<-open triangle 60 reversed}}
\tikzset{>=stealth}

\newcommand{\had}{\ar@[]}

\newcommand{\add}{\operatorname{add}}

\newcommand{\modendt}{\operatorname{mod}\operatorname{End}_{\cat}(T)^\text{op}}
\newcommand{\ccf}{\tc\ast\shift\tc}
\newcommand{\ideal}{(\tc^\perp)}
\newcommand{\rideal}{(\rc^\perp)}
\newcommand{\sideal}{(\sc^\perp)}

\newcommand{\cof}{\rightarrowtail}
\newcommand{\fib}{\twoheadrightarrow}
\newcommand{\htpcof}{\mathrel{\begin{tikzpicture}[baseline=-0.9mm] \draw (-145:0.1) --(0,0) --(145:0.1) --cycle ; \draw[thin] (0,0) --(0.3,0) ; \draw (.3,0) --++(145:0.1) ; \draw (0.3,0) --++(-145:0.1) ;\end{tikzpicture}}}
\newcommand{\tperp}{\tc^\perp}
\newcommand{\Cof}{\mathcal{C}of}
\newcommand{\Fib}{\mathcal{F}ib}

\newcommand{\sq}{\mathrel{\!\begin{tikzpicture} \draw (0,0) rectangle (0.2,0.2) ; \end{tikzpicture}\!}}
\newcommand{\hsq}{\mathrel{\begin{tikzpicture} \draw (0,0) rectangle (0.2,0.2) ; \fill (0,0) --(0.2,0.2) --(0,0.2) --cycle ;\end{tikzpicture}}}
\newcommand{\bsq}{\mathrel{\begin{tikzpicture} \draw (0,0) rectangle (0.2,0.2) ; \fill (0,0) --(0.2,0.2) --(0.2,0) --cycle ;\end{tikzpicture}}}

\newcommand{\ho}{\operatorname{Ho}}
\newcommand{\hoc}{\ho \cat}

\newcommand{\field}{\mathds{k}}

\title[From triangulated to module categories via homotopical algebra]{From triangulated categories to module categories via homotopical algebra}

\author[Palu]{Yann Palu}
\address{
LAMFA, Facult\'e des sciences, 
33 rue Saint-Leu, 
80039 Amiens Cedex 1,
FRANCE
}
\email{yann.palu@u-picardie.fr}

\begin{document}

\begin{abstract}
The category of modules over the endomorphism algebra of a rigid object
in a Hom-finite triangulated category $\cat$ has been given two different descriptions:
On the one hand, as shown by Osamu Iyama and Yuji Yoshino, it is equivalent to an ideal quotient of
a subcategory of $\cat$. On the other hand, Aslak Buan and Robert Marsh
proved that this module category is also equivalent to some localisation of $\cat$.
In this paper, we give a conceptual interpretation, inspired from homotopical algebra, of this double description.
Our main aim, yet to be acheived, is to generalise Buan-Marsh's result to the case of Hom-infinite cluster categories.
We note that, contrary to the more common case where a model category is a module category whose homotopy category
is triangulated, we consider here some triangulated categories whose homotopy categories are module categories.
\end{abstract}

\maketitle

\tableofcontents

\section*{Introduction}

Our aim in this paper is to give a homotopical algebraic interpretation of a result
of Aslak Buan and Robert Marsh \cite{BMloc1} (see also \cite{BMloc2, BeligiannisRigid})
on some localisations of triangulated categories associated
with rigid objects.

\subsection*{Endomorphism algebras of rigid objects}
The interest in rigid objects in module categories and in triangulated categories arose 
in tilting theory. The study of cluster algebras revived this interest: Rigid objects
in cluster categories categorify the cluster monomials of an associated cluster algebra.

Let $\cat$ be a triangulated category with suspension functor $\shift$.
An object $T$ in $\cat$ is called rigid if it has no non-trivial self-extensions, i.e. if $\cat(T,\shift T) = 0$.
A rigid object $T$ is called maximal rigid if moreover $T\oplus X$ is rigid if and only if $X\in\add T$,
where $\add T$ is the full subcategory of $\cat$ whose objects are the direct summands of finite directs sums of
copies of $T$. One nice feature of rigid objects in a triangulated category
is that all the information concerning their representation theory is contained
in the triangulated category:
\begin{theo}[Buan--Marsh--Reiten~\cite{BMR1}]\label{theorem: BMR1}
Let $\field$ be a field, let $Q$ be an acyclic quiver, and let $\cat = D^\text{b}(\field Q) / \tau^{-1}[1]$
be the associated cluster category~\cite{BMRRT}.
If $T$ is a maximal rigid object of $\cat$, then the functor $\cat(T,-)$ induces an equivalence
of categories: \[\cat / (\shift T) \stackrel{\simeq}{\gfl} \modendt,\]
where $(\shift T)$ denotes the ideal of morphisms factoring through $\add\shift T$.
\end{theo}
We note that, if $X\in\cat$, then the $\modendt$-module structure on the finite dimensional $\field$-vector space $\cat(T,X)$
is given by precomposition.  
Theorem~\ref{theorem: BMR1} was generalised in various directions; see for instance \cite{KZ}, \cite{KR1}, \cite{IY}.
In particular, it is shown in~\cite{KZ} that the abelian structure on $\cat_Q/(\shift T)$
can be described in terms of the triangulated structure~\cite{Kellertriorbcat} of $\cat_Q$.

Let us state a generalisation due to Osamu Iyama and Yuji Yoshino \cite{IY}, which was the starting point for
the work of Aslak Buan and Robert Marsh in \cite{BMloc1}.
\begin{theo}[Iyama--Yoshino~\cite{IY}]\label{theorem: IY}
Let $\field$ be a field and let $\cat$ be a $\field$-linear, Hom-finite, Krull--Schmindt,
triangulated category with some rigid object $T$. The full subcategory of $\cat$
whose objects are cones of morphisms in $\add T$ is denoted by $T\ast\shift T$.
Then the functor $\cat(T,-)$ induces an equivalence of categories:
\[T\ast\shift T / (\shift T) \stackrel{\simeq}{\gfl}\modendt.\]
\end{theo}

From this result arise the following questions: What are the properties of the functor $\cat(T,-) : \cat \fl \modendt$?
Is it possible to describe the module category $\modendt$ from $\cat$, without computing the subcategory $T\ast\shift T$?

The answer given by Aslak Buan and Robert Marsh is that $\cat(T,-)$ is a localisation functor.

\subsection*{Localisations}

The following situation arises in various fields of mathematics.
Assume that $\cat$ is a category with somme class $\wc$ of morphisms, called weak equivalences.
If there is a weak equivalence from $X$ to $Y$, one would like to think $X$ and $Y$ as being isomorphic.
For example, $\cat$ might be the category of complexes of modules over some ring, and $\wc$ the class
of quasi-isomorphisms (morphisms inducing isomorphisms on homologies).
Or $\cat$ might be the category of compactly generated (weak Hausdorff) topological
spaces, with $\wc$ the class of weak equivalences
(morphisms inducing bijections on homotopy groups, for all choices of a base point).

There is a method ~\cite{GabrielZisman} for constructing a new category $\cat[\wc^{-1}]$ having the same objects
as $\cat$ but where morphisms in $\wc$ become isomorphisms.

\begin{defi}
A localisation of $\cat$ at $\wc$ is the datum of a functor $\cat\stackrel{L}{\gfl}\cat[\wc^{-1}]$
with the property that, for any functor $\cat\stackrel{F}{\gfl}\dc$ such that $Fw$ is an isomorphism in $\dc$
whenever $w$ is in $\wc$, there is a unique functor $\cat[\wc^{-1}] \stackrel{G}{\gfl}\dc$ such that $GL = F$:
\[\xymatrix{
\cat \dr^F \bas_L & \dc \\
\cat[\wc^{-1}] \hdrp_G &
}\]
\end{defi}

We note that the diagram above is required to commute ``on the nose'' and not only up to some natural isomorphism.
In particular, the category $\cat[\wc^{-1}]$, if it exists, is unique up to isomorphism (and not just up to equivalence).

The receipe given in~\cite{GabrielZisman} for constructing $\cat[\wc^{-1}]$ can be sketched as follows: Consider all words
on (composable) morphisms of $\cat$ and formal inverses $w^{-1}$ to all morphisms $w$ in $\wc$, up to the equivalence relation
obtained by identifying subwords of the form $fg$ and $f\circ g$, $1f$ or $f1$ and $f$, and $ww^{-1}$ or $w^{-1}w$ and $1$.
The ``category'' with objects the objects of $\cat$, with morphisms the equivalence classes of words, and with composition
induced by concatenation of words is a localisation of $\cat$ at $\wc$. 
Unfortunately, there is some set-theoretic issue with this construction: The collection of all morphisms between two objects might
form a proper class rather than a set. As shown by Theorem~\ref{theorem: BMloc1},
this issue does not arise in the setup considered in~\cite{BMloc1}.

Let $\cat$ be, as in Theorem~\ref{theorem: IY}, a $\field$-linear, Hom-finite, Krull--Schmidt, triangulated category,
and let $T\in\cat$ be rigid. We write $T^\perp$ for the full subcategory of $\cat$
whose objects $X$ satisfy $\cat(T,X) = 0$.
Let $\sc$ be the class of morphisms $X\stackrel{f}{\gfl}Y$ such that, form some (equivalently, any)
triangle $Z\stackrel{g}{\gfl}X\stackrel{f}{\gfl}Y\stackrel{h}{\gfl}\shift Z$, both morphisms $g$ and $h$ belong to the ideal $(T^\perp)$
of morphisms factoring through $T^\perp$.

\begin{theo}[Buan--Marsh~\cite{BMloc1}]\label{theorem: BMloc1}
Let $\cat$ be a $\field$-linear, Hom-finite, Krull--Schmidt, triangulated category with a Serre functor,
and let $T\in\cat$ be rigid.
\begin{enumerate}
 \item For any morphism $s\in\sc$, $\cat(T,s)$ is an isomorphism in $\modendt$.
 \item The functor $\cat[\sc^{-1}] \stackrel{G}{\gfl}\modendt$ induced from $\cat(T,-)$ is an equivalence of categories.
\end{enumerate}
\end{theo}
In particular, the localisation of $\cat$ at $\sc$ exists: The construction of~\cite{GabrielZisman}
is a category.
A key lemma in the proof of Theorem~\ref{theorem: BMloc1} is:
\begin{lemma}[Buan--Marsh~\cite{BMloc1}]\label{lemma: key lemma}
Let $X\in\cat$. Then there is a triangle $Y\stackrel{g}{\gfl}A\stackrel{f}{\gfl}X\stackrel{h}{\gfl}\shift Y$, with
$A\in T\ast\shift T$, $Y\in T^\perp$ and $h\in(T^\perp)$.
In particular, the modules $\cat(T,A)$ and $\cat(T,X)$ are isomorphic.
\end{lemma}

Under the assumptions of Theorem~\ref{theorem: BMloc1}, we thus obtain two equivalent categories.
The first one is the localisation of $\cat$ at some class of morphisms $\sc$. The second one
is the full subcategory $T\ast\shift T$ of $\cat$ where morphisms are considered up to some equivalence relations
(two morphisms $f$ and $g$ are equivalent if $f-g$ factors through $\add\shift T$).
This is reminiscent to the theory of model categories~\cite{QuillenModel}.
Our aim is to make this analogy more precise: We will give some homotopical algebraic interpretation of Theorem~\ref{theorem: BMloc1},
and of Lemma~\ref{lemma: key lemma}. Our main motivation for pushing this analogy farther is the hope that
it might provide a tool allowing for a generalisation of Theorem~\ref{theorem: BMloc1}
including the case of Hom-infinite cluster categories (\cite{Acqw}, \cite{Plamondon1}).

\subsection*{Model categories}
Model categories, which axiomatise homotopy theory,
were introduced by Daniel G. Quillen in~\cite{QuillenModel}.
Let $\cat$ be a category and $\wc$ a collection of morphisms to be inverted.
If $(\cat,\wc)$ can be endowed with a model category structure, then the localisation,
called $\hoc$, of $\cat$ at $\wc$ exists (and comes equipped with more structure).

This axiomatic version of homotopy theory was called \emph{homotopical algebra}
by Daniel G. Quillen, since it subsumes both homological algebra
(when $\cat$ is the category of complexes of modules over a ring,
$\wc$ is the class of quasi-isomorphisms, and $\hoc$ is the derived category)
and homotopy
(e.g. when $\cat$ is the category of compactly generated topological spaces,
$\wc$ is the class of weak equivalences, and $\hoc$ is the homotopy category of spaces).

Assume that $\cat$ has finite limits and colimits (some authors assume all small limits and colimits).
Then a model category structure on $\cat$ is the datum of three classes $\wc, \Fib, \Cof$ of morphisms,
called respectively weak equivalences, fibrations and cofibrations,
satisfying some set of axioms inspired from basic homotopy theory. The first two axioms
concern the stability properties of $\wc, \Fib, \Cof$, and the other two axioms
ensure that the three classes interact nicely. More explicitely:
\begin{enumerate}
 \item The weak equivalences have the two-out-of-three property: For any composable $f$ and $g$, if
any two of $f,g$ and $gf$ are weak equivalences, then so is the third.
 \item The classes $\wc, \Fib$ and $\Cof$ contain all identities, are closed under compositions and under
taking retracts (in the category of morphisms of $\cat$).
 \item Lifting properties: $(\wc\cap\Cof)\, \square\, \Fib$ and $\Cof\, \square\, (\Fib\cap\wc)$.
 \item Factorisations: Any morphism belongs both to $\Fib\circ(\wc\cap\Cof)$
and to $(\Fib\cap\wc)\circ\Cof$.
\end{enumerate}
The notation $\square$ in Axiom (3) has the following meaning: Write $f\,\square\, g$ if,
for any commutative square
\[\xymatrix{
X \dr^a \bas_f & X' \bas^g \\
Y \dr_b \hdrp^\al & Y',
}\]
there is a lift $\al$ such that $\al f = a$ and $g\al = b$.
By Axiom (4), any morphism $f$ admits two factorisations:
\begin{center}
\begin{tikzpicture}
\draw (0,0) node{$X$} ;
\draw[->] (.3,0) --(1.7,0) node[midway, above]{$f$} ;
\draw (2,0) node{$Y$} ;
\draw (1,-.7) node{$X'$} ;
\draw[>->] (.2,-.2) --(.7,-.7) node[sloped, midway, above]{$\sim$} node[pos=.3, below]{$i$} ;
\draw[->>] (1.3,-.7) --(1.8,-.2) node[pos=.7, below]{$p$};
\begin{scope}[xshift=5cm]
\draw (0,0) node{$X$} ;
\draw[->] (.3,0) --(1.7,0) node[midway, above]{$f$} ;
\draw (2,0) node{$Y$} ;
\draw (1,-.7) node{$Y'$} ;
\draw[>->] (.2,-.2) --(.7,-.7) node[pos=.3, below]{$j$};
\draw[->>] (1.3,-.7) --(1.8,-.2) node[sloped, midway, above]{$\sim$} node[pos=.7, below]{$q$} ;
\end{scope}
\end{tikzpicture}
\end{center}
where $i$ is a cofibration and a weak equivalence, $p$ is a fibration, $j$ is a cofibration
and $q$ is a fibration and a weak equivalence.
An object $X$ is \emph{fibrant} if the canonical morphism from $X$ to the terminal object $\ast$ of $\cat$
(which exists since $\cat$ has finite limits) is a fibration.
Dually, $A$ is \emph{cofibrant} if the canonical morphism from the initial object $\emptyset$ to $A$ is a cofibration.
By applying Axiom (4) to $X\fl\ast$ and $\emptyset\fl X$, every object $X$ is seen to be weakly equivalent
to some fibrant object and to some cofibrant object. These are called fibrant and cofibrant replacements.
Let $\cat_{cf}$ be the full subcategory of $\cat$ whose objects are both fibrant and cofibrant.
In any model category, one can define paths objects, cylinder objects and homotopies
thus giving an axiomatic version of the corresponding notions for topological spaces
(see Section~\ref{ssection: w-model} for detailed definitions).
We write $f\simeq_\text{htp} g$ if two morphisms $f$ and $g$ are homotopic.
\begin{theo}[Quillen~ \cite{QuillenModel}]\label{theorem: QuillenModel}
Let $\cat$ be a model category and let $\hoc$ be the localisation of $\cat$
at the class of weak equivalences. Then:
\begin{itemize}
 \item[(i)] For any $X,Y\in\cat_{cf}$, homotopy is an equivalence relation on $\cat(X,Y)$,
 compatible with composition.
 \item[(ii)] The inclusion of $\cat_{cf}$ into $\cat$ induces an equivalence of categories
\[\cat_{cf} / \!\simeq_\text{htp} \;\gfl \hoc.\] In particular $\hoc$ is a category.
\end{itemize}
\end{theo}

There are two similar pictures coming from different setups in Theorem~\ref{theorem: BMloc1} and
Theorem~\ref{theorem: QuillenModel}:
\begin{center}
\begin{tikzpicture}
\draw (0,0) node{$\cat \ni T$ rigid} ;
\draw (-1.7,-2) node{$\frac{\ccf}{(\shift T)}$} ;
\draw (.7, -2) node{$\cat[\sc^{-1}]$} ;
\draw[->] (-1.1,-2) --(0,-2) node[midway, above]{$\simeq$} ;
\begin{scope}[shift={(-1,-.5)}, rotate=-120, scale=.8]
\draw (0,0) arc (0:180:.12 and .1) ;
\draw (.24,0) arc (0:-180:.12 and .1) ;
\draw (.48,0) arc (0:180:.12 and .1) ;
\draw (.72,0) arc (0:-180:.12 and .1) ;
\draw (.96,0) arc (0:180:.12 and .1) ;
\draw (1.2,0) arc (0:-180:.12 and .1) ;
\draw (1.32,.05) arc (90:180:.12 and .05) ;
\draw[->] (1.32,.05) --(1.5,.05) ;
\end{scope}
\begin{scope}[shift={(-.3,-.5)}, yscale=-1, rotate=60, scale=.8]
\draw (0,0) arc (0:180:.12 and .1) ;
\draw (.24,0) arc (0:-180:.12 and .1) ;
\draw (.48,0) arc (0:180:.12 and .1) ;
\draw (.72,0) arc (0:-180:.12 and .1) ;
\draw (.96,0) arc (0:180:.12 and .1) ;
\draw (1.2,0) arc (0:-180:.12 and .1) ;
\draw (1.32,.05) arc (90:180:.12 and .05) ;
\draw[->] (1.32,.05) --(1.5,.05) ;
\end{scope}
\begin{scope}[xshift=5cm]
\draw (0,0) node{$\cat$ model category} ;
\draw (-1.7,-2) node{$\cat_{cf}/\!\simeq_\text{htp}$} ;
\draw (.7, -2) node{$\cat[\wc^{-1}]$} ;
\draw[->] (-.8,-2) --(0,-2) node[midway, above]{$\simeq$} ;
\begin{scope}[shift={(-1,-.5)}, rotate=-120, scale=.8]
\draw (0,0) arc (0:180:.12 and .1) ;
\draw (.24,0) arc (0:-180:.12 and .1) ;
\draw (.48,0) arc (0:180:.12 and .1) ;
\draw (.72,0) arc (0:-180:.12 and .1) ;
\draw (.96,0) arc (0:180:.12 and .1) ;
\draw (1.2,0) arc (0:-180:.12 and .1) ;
\draw (1.32,.05) arc (90:180:.12 and .05) ;
\draw[->] (1.32,.05) --(1.5,.05) ;
\end{scope}
\begin{scope}[shift={(-.3,-.5)}, yscale=-1, rotate=60, scale=.8]
\draw (0,0) arc (0:180:.12 and .1) ;
\draw (.24,0) arc (0:-180:.12 and .1) ;
\draw (.48,0) arc (0:180:.12 and .1) ;
\draw (.72,0) arc (0:-180:.12 and .1) ;
\draw (.96,0) arc (0:180:.12 and .1) ;
\draw (1.2,0) arc (0:-180:.12 and .1) ;
\draw (1.32,.05) arc (90:180:.12 and .05) ;
\draw[->] (1.32,.05) --(1.5,.05) ;
\end{scope}
\end{scope}
\end{tikzpicture}
\end{center}

Our main result is motivated by this analogy.

\subsection*{Main result}
Let $\cat$ be a triangulated category and let $T\in\cat$ be rigid and contravariantly finite.
Let $\wc$ be the class of morphisms $X\stackrel{f}{\gfl}Y$ such that, for any
triangle $Z\stackrel{g}{\gfl}X\stackrel{f}{\gfl}Y\stackrel{h}{\gfl}\shift Z$, both morphisms $g$ and $h$ belong to the ideal $(T^\perp)$.

Consider $J:=\{0\fl\shift T\}$ and $I:=\bigcup\cat(R,A)$, where the union is taken over
a set of representatives for the isomorphism classes of objects $R\in\add T$ and $A\in T\ast\shift T$.
We define three classes of morphisms: $\Fib:=J^\square$, $w\Fib:=I^\square$ and $\Cof:=\prescript{\square}{}{w\Fib}$.
\begin{theo}\label{theorem: MainIntro}(See Theorem~\ref{theorem: main} for details)
Let $\cat$ be a triangulated category and let $T\in\cat$ be rigid and contravariantly finite. Then
the datum of ($\wc$, $\Fib$, $\Cof$) is \emph{almost} a model category structure on $\cat$.
Moreover:
\begin{itemize}
 \item[(i)] All objects are fibrant.
 \item[(ii)] An object is cofibrant if and only if it belongs to $T\ast\shift T$.
 \item[(iii)] Two morphisms of fibrant and cofibrant objects
 are homotopic if and only if their difference factors through $\add\shift T$.
\end{itemize}
\end{theo}
There are two reasons for the appearance of the term \emph{almost} in the statement above.
First, the category $\cat$ does not have finite (let alone all small)
limits and colomits in general: It only has finite direct sums, weak kernels and weak cokernels.
Second, every morphism can be factored out as a trivial cofibration followed by a fibration, but
the second factorisation only exists in a weakened form (see Definition~\ref{Definition: weak-model} for details).

\begin{cor}
The inclusion of $T\ast\shift T$ into $\cat$ induces an equivalence of categories
from $T\ast\shift T / (\shift T)$ to $\cat[\wc^{-1}]$. In particular, the localisation of $\cat$
at $\wc$ exists.
\end{cor}

\begin{rk}\rm
\begin{enumerate}
 \item In that setup, the morphism $A\fl X$ of Lemma~\ref{lemma: key lemma} can be interpreted as
a cofibrant replacement.
 \item In \cite[Lemma 4.4 and Theorem 4.6]{BeligiannisRigid},
Apostolos Beligiannis proved a generalised version of Theorem~\ref{theorem: BMloc1}, which applies to
contravariantly finite, rigid subcategories. This is also the generality in which we prove Theorem~\ref{theorem: main}.
\end{enumerate}
\end{rk}

After this rather long introduction, we describe the structure of the paper.
In Section~\ref{section: w-model}, we recall some elementary notions from homotopical algebra,
that we slightly modify in order to fit into the setup of Theorem~\ref{theorem: MainIntro}.
We define left-weak model categories and recall some basic definitions from homotopical algebra in Section~\ref{ssection: w-model}.
In Section~\ref{ssection: ppties}, we prove some elementary properties that hold in any left-weak model category.
We adapt the proof of Theorem~\ref{theorem: QuillenModel} to this setup in Section~\ref{ssection: homotopy}, and
consider generating cofibrations and generating trivial cofibrations in Section~\ref{ssection: cof gene}.
Section~\ref{section: w-model str from rigids} is dedicated to proving Theorem~\ref{theorem: MainIntro},
in the more general setup of rigid subcategories.

Throughout the paper, we will use the following conventions:
Products are denoted $\times$ and coproducts $\coprod$ or $\sqcup$.
The morphism $A\coprod B \fl C$ induced by $f$ on $A$ and $g$ on $B$ is written  $\left[f\,g\right]$.
The dual version, from $X$ to $Y\times Z$, is written $(f,g)$.
We write $f\coprod g$ for the morphism $A\coprod B \fl C\coprod D$ induced by
$A\stackrel{f}{\fl}C$ and $B\stackrel{g}{\fl}D$.
When direct sums are involved, we use matrix notations in order to describe morphisms.

\section{Left-weak model categories}\label{section: w-model}

In this section, we mostly recall the very basics on model categories.
We do so with a slightly modified definition of model categories, so as to
include the case of cluster categories, considered in Section~\ref{section: w-model str from rigids}.

References on model categories include the original \cite{QuillenModel},
and the books \cite{Hovey, HirschhornModel}.

\subsection{Definitions and notations from model category theory}\label{ssection: w-model}

For $f,g$ two morphisms in a category, we write $f\sq g$ if, for every commutative square
\[
\xymatrix{
A \dr \bas_f & \bas^g X \\
B \dr \hdrp & Y
}
\]
there is a lift (dashed arrow) where the two triangles commute.
If $f\sq g$, we say that $f$ has the left lifting property with respect to $g$, that
$g$ has the right lifting property with respect to $f$ or that $f$ and $g$ are weakly orthogonal.
We note that this is also equivalent to the statement: Every morphism from $f$ to $g$ is null-homotopic.

If $\ac$ is a collection of morphisms in a category $\cat$, we let $^\square\ac$, resp. $\ac^\square$,
denote the collection of those morphisms that have the left, resp. right, lifting property with respect
to all morphsisms in $\ac$.

Let $\cat$ be a category endowed with a class of morphisms $\wc$. The morphisms in $\wc$
are morphisms that we would like to invert. Let $\Cof$ and $\Fib$ be two classes of morphisms of $\cat$.
If a morphism $w$ belongs to $\wc$, we write $X \stackrel{\sim}{\fl} Y$, and call it a weak equivalence.
If a morphism $i$ belongs to $\Cof$, we write $X\stackrel{i}{\cof}Y$, and call it a cofibration.
If a morphism $p$ belongs to $\Fib$, we write $X\stackrel{p}{\fib}Y$, and call it a fibration.
A morphism which is both a weak equivalence and a cofibration is called a weak, trivial or acyclic cofibration,
and similarly for fibrations.

\subsubsection*{Cylinder objects and left homotopies} Assume that $\cat$ has finite coproducts, and let $X$ be an object of $\cat$.
The morphism $X\coprod X \fl X$ induced by the identity on each copy of $X$ is denoted by $\nabla$.
A \emph{cylinder object} for $X$ is a factorisation of $\nabla$ as follows: $X\coprod X \fl X' \stackrel{\sim}{\fl} X$.
If $f$ and $g$ are two morphisms from $X$ to $Y$ in $\cat$, a \emph{left homotopy} from $f$ to $g$ is a morphism
$H$ from a cylinder object $X'$ for $X$ to $Y$ such that the composition $X\coprod X \fl X' \stackrel{H}{\fl} X$ is
the morphism induced by $f$ on the first copy of $X$, and by $g$ on the second copy of $X$.
Most of the proofs involving left homotopies are easily understood if one has the following picture in mind:

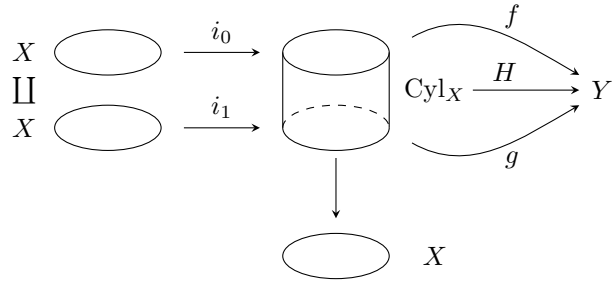
\begin{center}
\begin{tikzpicture}
\draw (-1.1, 0) node{$X$} ;
\draw (-1.1,-.5) node{$\coprod$} ;
\draw (-1.1,-1) node{$X$} ;
\draw (0,0) circle (.7 and .3) ;
\draw[->] (1,0) --(2,0) node[midway, above]{$i_0$} ;
\draw (0,-1) circle (.7 and .3) ;
\draw[->] (1,-1) --(2,-1) node[midway, above]{$i_1$} ;
\draw (3,0) circle (.7 and .3) ;
\draw (2.3,-1) arc (-180:0:.7 and .3) ;
\draw[dashed] (2.3,-1) arc (180:0:.7 and .3) ;
\draw (2.3,0) --(2.3,-1) ;
\draw (3.7,0) --(3.7,-1) ;
\draw (4.3,-.5) node{Cyl$_X$} ;
\draw[->] (4.8,-.5) --(6.2,-.5) node[pos=.3, above]{$H$} ;
\draw (6.5,-.5) node{$Y$} ;
\draw (4,0.2) edge[->, out=30, in=150] node[pos=.6, above]{$f$} (6.2,-.3) ;
\draw (4,-1.2) edge[->, out=-30, in=-150] node[pos=.6, below]{$g$} (6.2, -.7) ;
\draw[->] (3,-1.4) --(3,-2.2) ;
\draw (3,-2.7) circle (.7 and .3) ;
\draw (4.3,-2.7) node{$X$} ;
\end{tikzpicture}
\end{center}

The dual notions are that of:

\subsubsection*{Path objects and right homotopies} Assume that $\cat$ has finite products, and let $Y$ be an object of $\cat$.
A \emph{path object} for $Y$ is a factorisation of the diagonal morphism $Y\stackrel{\Delta}{\fl} Y\times Y$ as follows:
$Y \stackrel{\sim}{\fl} Y' \fl Y\times Y$.
If $f$ and $g$ are two morphisms from $X$ to $Y$ in $\cat$, a \emph{right homotopy} from $f$ to $g$ is a morphism
$K$ from $X$ to a path object $Y'$ for $Y$ such that the composition $X \stackrel{K}{\fl} Y' \fl Y\times Y$ is
the morphism induced by $f$ to the first copy of $Y$, and by $g$ to the second copy of $Y$.
If there is a right homotopy from $f$ to $g$, we write $f\sim^r g$.
A particular case that one might have in mind comes from the category of (unbased) topological spaces, where
$Y'$ is the space of paths in $Y$, the map $Y \fl Y'$ maps a point $y$ to the constant path at $y$, and the map
$Y'\fl Y\times Y$ sends a path to the pair formed by its starting point and its ending point.

Two morphisms $f,g$ are \emph{homotopic}, written $f\sim g$, if they are both left and right homotopic.
A morphism $f$ is called a \emph{homotopy equivalence}
if there is some morphism $g$ such that $fg$ and $gf$ are homotopic to the respective identity morphisms.

Assume that $\cat$ has an initial object $\emptyset$ and a final object $\ast$.
An object $X$ of $\cat$ is called \emph{fibrant} if the canonical morphism $X \fl \ast$ is a fibration.
An object $B$ is called \emph{cofibrant} if the canonical morphism $\emptyset \fl B$ is a cofibration.

For two morphisms $f$ and $g$, we write $f\hsq g$ (resp. $f\bsq g$) if for every commutative square
\[
\xymatrix{
A \dr^u \bas_f & X \bas^g \\
B \dr_v \hdrp^\al & Y
}
\]
there is a lift $\al$ such that $g\al=v$ and $\al f \sim^r u$
(resp. $g\al\sim^l v$ and $\al f = u$). (The filled in triangle in the symbols is meant to represent a 2-cell).

\begin{defi}\rm
A morphism $A\stackrel{j}{\fl}B$ is called a \emph{homotopy cofibration}, denoted $\htpcof$,
if $j\hsq(\wc\cap\Fib)$ and $j\bsq(\wc\cap\Fib)$.
\end{defi}

The \emph{homotopy category} of $\cat$ is the localisation $\cat[\wc^{-1}]$ at the class of weak equivalences. It
is denoted $\hoc$. The assumptions below will ensure that this localisation is well-defined.

\begin{defi}\label{Definition: weak-model}\rm
We call a category $\cat$ a left-weak model category if it
is equipped with three classes of morphisms $\wc$, $\Cof$ and $\Fib$ and if the following axioms are satisfied:
\end{defi}
\begin{itemize}
 \item[(0)] The category $\cat$ has finite products and coproducts, (an initial object $\emptyset$ and a final object $\ast$).
Pull-backs of trivial fibrations along epimorphisms exist and are trivial fibrations.
 \item[(0.5)] If $A$ is cofibrant, then for any homotopy cofibration $A\htpcof B$, B is cofibrant.
 If $B$ is a cofibrant object, then the canonical inclusions $A \fl A\coprod B$ and $A \fl B\coprod A$ are cofibrations.
 \item[(1)] The class $\wc$ has the two-out-of-three property: If $f,g$ are composable morphisms and if any two
of $f,g,gf$ are weak-equivalences, then so is the third.
 \item[(2)] The classes $\wc,\Cof,\Fib$ contain all identities, are stable under composition and under retracts.
 \item[(3)] The following weak orthogonality relations hold:
 $\left(\wc\cap\Cof\right)\subseteq\prescript{\square}{}{\Fib}$ and $\Cof\subseteq\prescript{\square}{}{\left(\Fib\cap\wc\right)}$.
 \item[(4)] Every morphism $X\stackrel{f}{\gfl}Y$ admits a factorisation
as a weak cofibration followed by a fibration.
Every morphism $A\stackrel{g}{\gfl}B$ with $A$ cofibrant admits a factorisation
as a homotopy cofibration followed by a weak fibration.
\begin{center}
\begin{tikzpicture}
\draw (0,0) node{$X$} ;
\draw[->] (.3,0) --(1.7,0) node[midway, above]{$f$} ;
\draw (2,0) node{$Y$} ;
\draw (1,-.7) node{$X'$} ;
\draw[>->] (.2,-.2) --(.7,-.7) node[sloped, midway, below]{$\sim$} ;
\draw[->>] (1.3,-.7) --(1.8,-.2) ;
\begin{scope}[xshift=5cm]
\draw (0,0) node{$A$} ;
\draw[->] (.3,0) --(1.7,0) node[midway, above]{$g$} ;
\draw (2,0) node{$B$} ;
\draw (1,-.7) node{$B'$} ;
\draw[htpcof, thick] (.7,-.7) --(.2,-.2) ;
\draw[->>] (1.3,-.7) --(1.8,-.2) node[sloped, midway, below]{$\sim$} ;
\end{scope}
\end{tikzpicture}
\end{center}

\end{itemize}

\begin{rk}\rm
These are almost the same axioms as those for a model category, with two noticeable differences.
First, the category $\cat$ is not assumed to have finite limits and colimits. This is because
we want to include the case when the category $\cat$ is a cluster category.
Second, in a model category, every morphism can be written as a cofibration followed by a weak fibration. This axiom is weakened
here because in the example studied in Section~\ref{section: w-model str from rigids}, such a factorisation might not exist.
\end{rk}

\begin{rk}\rm
By analogy with model categories, the properties listed in Axiom 0.5 should be easy consequences of the other axioms.
This is not true here, mainly because of the weakened factorisation in Axiom 4.
\end{rk}

\begin{rk}\rm
For model categories, the factorisations are often assumed functorial.
This is not the case in the main example considered in that paper, therefore
no such assumption is made in the definition. However, the factorisations constructed in Section~\ref{section: w-model str from rigids}
are functorial at the level of the homotopy category.
\end{rk}

\begin{rk}\rm
As stated, the axioms are not self-dual. There is an obvious definition for what might be a weak model category:
weaken Axiom 4 for both factorisations, and stengthen Axiom 0.5 accordingly. In that case, we do not know whether
homotopy remains an equivalence relation. What can easily been done is to require that $Y$ be fibrant for the first factorisation
of Axiom 4 to exist. With that definition, the dual of a left weak model category becomes a ``right weak model category''. All the results
remain valid, but some statements have to be modified as follows:
\begin{itemize}
 \item In Lemma~\ref{lemma: first ppties}, (c) only holds for morphisms with fibrant codomain. As a consequence,
 (d) becomes: The pull-back of a fibration along a morphism with fibrant domain is a fibration. Moreover, (e) might not hold
 (the proofs that follow do not make use of this property).
 \item In Lemma~\ref{lemma: specific homotopies} (1), $Y$ has to be assumed fibrant (as is the case in all proofs using this lemma).
 \item In Lemma~\ref{lemma: htp and composition} (b), $Y$ is assumed fibrant. The only consequence of this restriction is the next point.
 \item In Proposition~\ref{proposition: whtpcof implies htpwcof}, $B$ has to be assumed fibrant. Moreover, in point (2) of this proposition,
 the morphism $f$ only has the left lifting property $\hsq$ with respect to fibrations with fibrant domains. This always holds when
 Proposition~\ref{proposition: whtpcof implies htpwcof} is used (In the proof of Proposition~\ref{proposition: h ast bijective},
 the same restriction is required on $h$, but causes no trouble. In the proof of Proposition~\ref{lemma: lhtp iff rhtp}, the object $X'$
 of the last diagram is fibrant, and in the proof of Proposition~\ref{proposition: weq and htpeq}, the object $Z$ is fibrant).
 \item Proposition~\ref{proposition: cof gene} holds with few modifications. Condition (4) becomes: Any morphism $f$ in $\Cof\cap\wc$
 belongs to $w\Cof$ and the converse holds if $f$ has fibrant codomain. Condition (5) only concerns morphisms with fibrant codomains.
\end{itemize}
Moreover, one can check that, in each proof that we have not mentioned and where the first
factorisation is used, the codomains involved are fibrant. This suggests that a better definition for left weak model categories
would require $Y$ to be fibrant. We haven't done so here in order to simplify the exposition and because, in the example
of a left weak model category studied in Section~\ref{section: w-model str from rigids}, all objects are fibrant.
\end{rk}

\subsection{Some basic properties}\label{ssection: ppties}

In this section, we fix a left-weak model category $(\cat,\wc,\Cof,\Fib)$.

In Lemma~\ref{lemma: first ppties}, we collect many basic properties
of model categories that also hold, with the same proofs,
for left-weak model categories.

\begin{lemma}\label{lemma: first ppties}
\begin{enumerate}[(a)]
 \item Left$-\sq$ are stable under pull-backs.
 \item Right$-\sq$ are stable under push-outs.
 \item $(\wc\cap\Cof)^\square = \Fib$ and $\wc\cap\Cof = \prescript{\square}{}{\Fib}$.
 \item $\Fib$ is stable under pull-backs.
 \item $\wc\cap\Cof$ is stable under push-outs.
 \item The initial object is cofibrant and the terminal object is fibrant.
 \item An object is fibrant if and only if it is injective relative to  weak cofibrations.
 \item Any cofibrant object is projective relative to weak fibrations.
 \item $f\sim^l g$ implies $hf\sim^l hg$.
 \item $f\sim^r g$ implies $fh\sim^r gh$.
 \item If $X$ is fibrant and $Y$ is any object, then the two projections $X\times Y \fl X$ and
 $Y\times X \fl X$ are fibrations.
\end{enumerate}
\end{lemma}

\begin{proof}
We only give a hint for (c), since the other points are straightforward.
If $f$ has the right lifting property with respect to trivial cofibrations,
one uses the first factorisation of Axiom (4) in order to show that $f$ is the retract of a fibration.
\end{proof}

\begin{lemma}\label{lemma: composition htpcof}
Let $f\in\cat(Y,Z)$ and $g\in\cat(X,Y)$ be composable morphisms.
If $f$ is a homotopy cofibration and $g$ a cofibration, then $fg$ is a homotopy cofibration. 
\end{lemma}

\begin{proof}
Let us check that $fg$ satisfies the two lifting properties defining homotopy cofibrations.
Consider a commutative diagram:

\begin{center}
\begin{tikzpicture}[>=stealth]
\coordinate (X) at (0,2) ;
\coordinate (Y) at (0,1) ;
\coordinate (Z) at (0,0) ;
\coordinate (A) at (2,2) ;
\coordinate (B) at (2,0) ;
\coordinate (U) at (0,0.2) ;
\coordinate (D) at (0,-0.2) ;
\coordinate (L) at (-0.2,0) ;
\coordinate (R) at (0.2,0) ;
\draw (X) node{$X$} ;
\draw (Y) node{$Y$} ;
\draw (Z) node{$Z$} ;
\draw (A) node{$A$} ;
\draw (B) node{$B$} ;
\draw[>->] ($(X)+(D) $) -- ($(Y)+(U) $) node[midway, left]{$g$} ;
\draw[<-open triangle 60 reversed] ($(Z)+(U)$) --($(Y)+(D)$) node[midway, left]{$f$} ;
\draw[->] ($(X)+(R)$) --($(A)+(L)$) node[midway, above]{$a$} ;
\draw[->>] ($(A)+(D)$) --($(B)+(U)$) node[midway, right]{$\wr\;\;p$} ;
\draw[->] ($(Z)+(R)$) --($(B)+(L)$) node[midway, below]{$b$} ;
\draw[->, dashed] ($(Y)+(U)+(R)$) --($(A)+(D)+(L)$) node[near start, above]{$c$} ;
\draw[->,dotted] ($(Z)+(U)+(R)$) --($(A)+(D)+(L)+(D)$) node[midway, above]{$d$} node[midway, below right]{$\sim^l$} ;
\begin{scope}[xshift=5cm]
\coordinate (X) at (0,2) ;
\coordinate (Y) at (0,1) ;
\coordinate (Z) at (0,0) ;
\coordinate (A) at (2,2) ;
\coordinate (B) at (2,0) ;
\coordinate (U) at (0,0.2) ;
\coordinate (D) at (0,-0.2) ;
\coordinate (L) at (-0.2,0) ;
\coordinate (R) at (0.2,0) ;
\draw (X) node{$X$} ;
\draw (Y) node{$Y$} ;
\draw (Z) node{$Z$} ;
\draw (A) node{$A$} ;
\draw (B) node{$B$} ;
\draw[>->] ($(X)+(D) $) -- ($(Y)+(U) $) node[midway, left]{$g$} ;
\draw[<-open triangle 60 reversed] ($(Z)+(U)$) --($(Y)+(D)$) node[midway, left]{$f$} ;
\draw[->] ($(X)+(R)$) --($(A)+(L)$) node[midway, above]{$a$} ;
\draw[->>] ($(A)+(D)$) --($(B)+(U)$) node[midway, right]{$\wr\;\;p$} ;
\draw[->] ($(Z)+(R)$) --($(B)+(L)$) node[midway, below]{$b$} ;
\draw[->, dashed] ($(Y)+(U)+(R)$) --($(A)+(D)+(L)$) node[near start, above]{$c$} ;
\draw[->,dotted] ($(Z)+(U)+(R)$) --($(A)+(D)+(L)+(D)$) node[midway, below]{$e$} node[midway, left]{$\sim^r$} ;
\end{scope}
\end{tikzpicture}
\end{center}

where $p$ is a trivial fibration. Since $g$ is a cofibration, there is a morphism $c$
such that $cg=a$ and $pc=bf$. Since $f$ is a homotopy cofibration, there are two morphisms
$d,e$ such that $df=c$, $pd\sim^l b$, $pe=b$ and $ef\sim^r c$. This implies
$dfg=cg=a$ and, by Lemma~\ref{lemma: first ppties} $(j)$, $efg \sim^r cg = a$, so that $fg$ is a homotopy cofibration.
\end{proof}

\begin{lemma}\label{lemma: i0 htpcof}
Let $B$ be cofibrant and let $B\coprod B \stackrel{[i_0\,i_1]}{\htpcof} B' \stackrel{\sim}{\gfl} B$ be a cylinder object for $B$.
Then the morphism $B \stackrel{i_0}{\gfl}B'$ is both a weak equivalence and a homotopy cofibration.
\end{lemma}

\begin{proof}
The morphism $B\stackrel{\left[^1_0\right]}{\gfl}B\coprod B$ is a cofibration,
by Axiom 0.5. The following commutative diagram:

\begin{center}
\begin{tikzpicture}[scale=2, >=stealth]
\draw (0,1) node{$B$} (0,0) node{$B\coprod B$} (1,0) node{$B'$} (2,0) node{$B$} ;
\draw[>->] (0,0.8) --(0,0.2) ;
\draw[<-open triangle 60 reversed] (0.8,0) --(0.3,0) node[midway, below]{$[i_0\,i_1]$} ;
\draw[->] (0.2,0.8) --(1,0.2) node[midway, right]{$i_0$} ;
\draw[->] (1.2,0) --(1.8,0) node[midway, above]{$\sim$} ;
\draw (0.3,1) edge[->, out=0, in=120] node[near end, above]{$1$}  (1.9,0.2) ;
\end{tikzpicture}
\end{center}
thus shows that the morphism $i_0$ is a homotopy cofibration by Lemma~\ref{lemma: composition htpcof},
and a weak equivalence by the two-out-of-three property.
\end{proof}

\begin{lemma}\label{lemma: specific homotopies}
Let $X,Y\in\cat$ and let $f,g\in\cat(X,Y)$.
\begin{enumerate}
 \item If $f\sim^l g$ via a path object $Y \stackrel{s}{\gfl} Y' \stackrel{p}{\gfl} Y\times Y$,
where $s$ is a weak equivalence, then $p$ can be assumed to be a fibration.
If moreover $X$ is cofibrant, then $s$ can be assumed to be a trivial cofibration.
 \item Assume that $f\sim^l g$ via a cylinder object
$X\coprod X \stackrel{i}{\gfl}X'\stackrel{t}{\gfl} X$, where $t$ is a weak equivalence.
If $Y$ is fibrant, then $t$ can be assumed to be a trivial fibration.
If $X$ is cofibrant, then $i$ can be assumed to be a homotopy cofibration.
\end{enumerate}
\end{lemma}

\begin{proof}
The proof of \emph{(1)} is the same as in the case of model categories.
The proof of \emph{(2)} is similar: Let $X'\stackrel{H}{\gfl}Y$ be a left homotopy from $f$ to $g$.
If $X$ is cofibrant, then $X\coprod X$ is also cofibrant by Axiom 0.5.
The morphism $i$ can thus be factored as a homotopy cofibration $j$ followed by a
weak fibration $p$. Then $Hp$ is a left homotopy from $f$ to $g$.
Assume $Y$ fibrant, and factor $t$ as a trivial cofibration $c$ followed by a trivial fibration $q$.
Since $Y$ is fibrant, the homotopy $H$ lifts through $c$ to a homotopy $H'$ from $f$ to $g$.
\end{proof} 

\begin{rk}\rm
If $X$ is cofibrant and $Y$ is fibrant in \emph{(2)}, then
$t$ can be assumed to be a trivial fibration and $i$ a homotopy cofibration at the same time.
One has to be careful and factor $t$ as a trivial cofibration $j$ followed by a trivial fibration $p$
first, and then factor $ji$ as a homotopy cofibration followed by a trivial fibration.
\end{rk}

Next proposition, though very easy, is the key to making things work in the setup
of left-weak model categories.
\begin{prop}\label{proposition: whtpcof implies htpwcof}
Let $A\stackrel{f}{\gfl}B$ be both a weak equivalence and a homotopy cofibration. Then:
\begin{enumerate}
 \item The lifting property $f\bsq\Fib$ holds;
 \item If moreover $A$ is cofibrant, the lifting property $f\hsq\Fib$ holds.
\end{enumerate}
\end{prop}

\begin{proof}
Assume that $f$ is a weak equivalence and a homotopy cofibration. Factor $f$ as
a trivial cofibration $i$ followed by a fibration $p$:

\begin{center}
\begin{tikzpicture}[>=stealth]
\coordinate (A) at (0,0) ;
\coordinate (B) at (2.5,0) ;
\coordinate (C) at (1.25,-1) ;
\coordinate (U) at (0,0.2) ;
\coordinate (D) at (0,-0.3) ;
\coordinate (L) at (-0.2,0) ;
\coordinate (R) at (0.3,0) ;
\draw (A) node{$A$} (B) node{$B$} (C) node{$C$} ;
\draw[<-open triangle 60 reversed]  ($(B)+(L)$) --($(A)+(R)$) node[pos=0.4, above]{$\sim$} node[pos=0.7, above]{$f$} ;
\draw[>->] ($(A)+(D)+(R)$) --($(C)+(U)+(L)$) node[sloped, midway, above]{$\sim$} node[midway, below left]{$i$} ;
\draw[->>] ($(C)+(U)+(R)$) --($(B)+(D)+(L)$) node[midway, below right]{$p$} ;
\end{tikzpicture}
\end{center}

By the two-out-of-three property, $p$ is a weak equivalence, so that there are morphisms $u$, $v$,
as in the diagrams:

\begin{center}
\begin{tikzpicture}[>=stealth]
\coordinate (A) at (0,2) ;
\coordinate (Bl) at (0,0) ;
\coordinate (C) at (2,2) ;
\coordinate (Br) at (2,0) ;
\coordinate (U) at (0,0.3) ;
\coordinate (D) at (0,-0.3) ;
\coordinate (L) at (-0.3,0) ;
\coordinate (R) at (0.3,0) ;
\draw (A) node{$A$} (Bl) node{$B$} (C) node{$C$} (Br) node{$B$} ;
\draw[<-open triangle 60 reversed] ($(Bl)+(U)$) --($(A)+(D)$) node[midway, left]{$f$} node[sloped, midway, below]{$\sim$};
\draw[>->] ($(A)+(R)$) --($(C)+(L)$) node[midway, above]{$i$} node[midway, below]{$\sim$} ;
\draw[->>] ($(C)+(D)$) --($(Br)+(U)$) node[midway, right]{$\wr\;\;p$} ;
\draw[double distance=1.5pt] ($(Bl)+(R)$) --($(Br)+(L)$) ;
\draw[->,dotted] ($(Bl)+(U)+(R)$) --($(C)+(L)+(D)$) node[pos=.7, left]{$u$} node[midway, below right]{$\sim^l$} ;
\begin{scope}[xshift=5cm]
\coordinate (A) at (0,2) ;
\coordinate (Bl) at (0,0) ;
\coordinate (C) at (2,2) ;
\coordinate (Br) at (2,0) ;
\coordinate (U) at (0,0.3) ;
\coordinate (D) at (0,-0.3) ;
\coordinate (L) at (-0.3,0) ;
\coordinate (R) at (0.3,0) ;
\draw (A) node{$A$} (Bl) node{$B$} (C) node{$C$} (Br) node{$B$} ;
\draw[<-open triangle 60 reversed] ($(Bl)+(U)$) --($(A)+(D)$) node[midway, left]{$f$} node[sloped, midway, below]{$\sim$};
\draw[>->] ($(A)+(R)$) --($(C)+(L)$) node[pos=.2, above]{$i$} node[pos=.6, above]{$\sim$} ;
\draw[->>] ($(C)+(D)$) --($(Br)+(U)$) node[midway, right]{$\wr\;\;p$} ;
\draw[double distance=1.5pt] ($(Bl)+(R)$) --($(Br)+(L)$) ;
\draw[->,dotted] ($(Bl)+(U)+(R)$) --($(C)+(L)+(D)$) node[midway, below]{$v$} node[midway, above =8pt]{$\sim^r$} ;
\end{scope}
\end{tikzpicture}
\end{center}
such that $pu\sim^l 1$, $uf=i$, $pv=1$ and $vf\sim^r i$.
Let $q$ be any fibration and let $(a,b)$ be any morphism from $f$ to $q$.
Since $i$ is a trivial cofibration, there is a lift $g$:

\begin{center}
\begin{tikzpicture}[>=stealth]
\coordinate (Al) at (0,0) ;
\coordinate (Ar) at (2,0) ;
\coordinate (X) at (4,0) ;
\coordinate (C) at (0,-2) ;
\coordinate (B) at (2,-2) ;
\coordinate (Y) at (4,-2) ;
\coordinate (U) at (0,0.3) ;
\coordinate (D) at (0,-0.3) ;
\coordinate (L) at (-0.3,0) ;
\coordinate (R) at (0.3,0) ;
\draw (Al) node{$A$} (Ar) node{$A$} (X) node{$X$} (C) node{$C$} (B) node{$B$} (Y) node{$Y$} ;
\draw[double distance=1.5pt] ($(Al)+(R)$) --($(Ar)+(L)$) ;
\draw[->] ($(Ar)+(R)$) --($(X)+(L)$) node[midway, above]{$a$} ;
\draw[htpcof] ($(B)+(U)$) --($(Ar)+(D)$) node[midway, right]{$f$} ;
\draw[>->] ($(Al)+(D)$) --($(C)+(U)$) node[midway, left]{$i$} node[midway, right]{$\wr$} ;
\draw[->>] ($(X)+(D)$) --($(Y)+(U)$) node[midway, right]{$q$} ;
\draw[->] ($(C)+(R)$) --($(B)+(L)$) node[midway, below]{$p$} ;
\draw[->] ($(B)+(R)$) --($(Y)+(L)$) node[midway, below]{$b$} ;
\draw ($(C)+(U)+(R)$) edge[->, loosely dashed, out=30, in=190] node[near start, above]{$g$} ($(X)+(D)+(L)$) ;
\end{tikzpicture}
\end{center}

such that $qg=bp$ and $gi=a$.
We thus have: $guf = gi = a$, $qgv = bpv = b$ and, by Lemma~\ref{lemma: first ppties} (i), $qgu = bpu \sim^l b$.
If $A$ is cofibrant, then Lemma~\ref{lemma: htp and composition}\,(b) implies $gvf \sim^r gi = a$ (we note that the proof
of Lemma~\ref{lemma: htp and composition} does not make use of Proposition~\ref{proposition: whtpcof implies htpwcof}).

\begin{center}
\begin{tikzpicture}[>=stealth]
\coordinate (A) at (0,2) ;
\coordinate (Bl) at (0,0) ;
\coordinate (C) at (2,2) ;
\coordinate (Br) at (2,0) ;
\coordinate (U) at (0,0.3) ;
\coordinate (D) at (0,-0.3) ;
\coordinate (L) at (-0.3,0) ;
\coordinate (R) at (0.3,0) ;
\draw (A) node{$A$} (Bl) node{$B$} (C) node{$X$} (Br) node{$Y$} ;
\draw[<-] ($(Bl)+(U)$) --($(A)+(D)$) node[midway, left]{$f$} ;
\draw[->] ($(A)+(R)$) --($(C)+(L)$) node[midway, above]{$a$} ;
\draw[->>] ($(C)+(D)$) --($(Br)+(U)$) node[midway, right]{$q$} ;
\draw[->] ($(Bl)+(R)$) --($(Br)+(L)$) node[midway, below]{$b$} ;
\draw[->,dotted] ($(Bl)+(U)+(R)$) --($(C)+(D)+(L)$) node[midway, above]{$gu\;$} node[midway, below right]{$\sim^l$} ;
\begin{scope}[xshift=5cm]
\coordinate (A) at (0,2) ;
\coordinate (Bl) at (0,0) ;
\coordinate (C) at (2,2) ;
\coordinate (Br) at (2,0) ;
\coordinate (U) at (0,0.3) ;
\coordinate (D) at (0,-0.3) ;
\coordinate (L) at (-0.3,0) ;
\coordinate (R) at (0.3,0) ;
\draw (A) node{$A$} (Bl) node{$B$} (C) node{$X$} (Br) node{$Y$} ;
\draw[<-] ($(Bl)+(U)$) --($(A)+(D)$) node[midway, left]{$f$} ;
\draw[->] ($(A)+(R)$) --($(C)+(L)$) node[midway, above]{$a$} ;
\draw[->>] ($(C)+(D)$) --($(Br)+(U)$) node[midway, right]{$q$} ;
\draw[->] ($(Bl)+(R)$) --($(Br)+(L)$) node[midway, below]{$b$} ;
\draw[->,dotted] ($(Bl)+(U)+(R)$) --($(C)+(D)+(L)$) node[midway, below]{$\;gv$} node[pos=.4, above =8pt]{$\sim^r$} ;
\end{scope}
\end{tikzpicture}
\end{center}
\end{proof}

\subsection{The homotopy category}\label{ssection: homotopy}

In this section, we prove Theorem~\ref{theorem: Quillen}, due to Quillen~\cite{QuillenModel}, in
the setup of left weak model categories. We follow the proof in Hovey~\cite{Hovey}, with only a few minor modifications,
the main one being that we do not prove left homotopy to be an equivalence relation. We merely prove that
right homotopy is, and that left homotopy and right homotopy coincide on $\cat(B,X)$ if $B$ is cofibrant and $X$ is fibrant.

The whole of this section might have been summed up into one sentence:
``It is easily checked that the proof of the theorem of Quillen generalises to the setup of
left-weak model categories''. We note that the reason why all proofs are included is,
on the one hand so as to make explicit these minor modifications,
and on the other hand for the potentially interested reader who has very few background knowledge on model categories
(as is the case of the author).

Recall that $\hoc$ denotes the localisation $\cat[\wc^{-1}]$ of $\cat$ at the class of weak equivalences.

\begin{theo}[Quillen]\label{theorem: Quillen}
\begin{enumerate}
 \item If $B$ is cofibrant and $X$ is fibrant, then homotopy is an equivalence relation on $\cat(B,X)$.
 \item Homotopy is compatible with composition in $\cat_{cf}$.
 \item The inclusion $\cat_{cf}\hookrightarrow\cat$ induces an equivalence of categories
 $\cat_{cf}/\!\!\sim\,\stackrel{\simeq}{\gfl}\hoc$. In particular, the localisation $\hoc$ is a well-defined category.
\end{enumerate}
\end{theo}

The remaining of this section is devoted to proving Theorem~\ref{theorem: Quillen}.

\begin{lemma}\label{lemma: htp and composition}
Let $f,g\in\cat(B,X)$, $h\in\cat(A,B)$ and $k\in\cat(X,Y)$.
\begin{itemize}
 \item[(a)] If $X$ is fibrant and $A$ cofibrant, then $f\sim^l g$ implies $fh \sim^l gh$.
 \item[(b)] If $B$ is cofibrant then $f\sim^r g$ implies $kf \sim^r kg$.
\end{itemize}
\end{lemma}

\begin{proof}
We only prove (a) since the usual proof of (b) applies as such to our setup.
Assume $X$ fibrant, $A$ cofibrant and $f\sim^l g$. By Lemma~\ref{lemma: specific homotopies},
there is a cylinder object $B\coprod B \fl B' \stackrel{\sim}{\fib} B$ and a homotopy
$B'\stackrel{H}{\gfl}X$ from $f$ to $g$.
Since $A$ is cofibrant, so is $A\coprod A$ (by Axiom 0.5) and we can factor $A\coprod A \stackrel{\nabla}{\gfl}A$
as a homotopy cofibration followed by a trivial fibration. We thus have commutative diagrams:
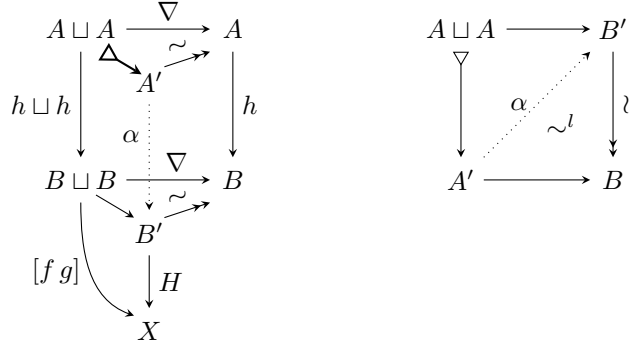
\begin{center}
\begin{tikzpicture}
\coordinate (U) at (0,0.3) ;
\coordinate (D) at (0,-0.3) ;
\coordinate (L) at (-0.3,0) ;
\coordinate (R) at (0.3,0) ;
\coordinate (aa) at (0,2) ;
\coordinate (a') at (0.9,1.3) ;
\coordinate (a) at (2,2) ;
\coordinate (bb) at (0,0) ;
\coordinate (b') at (0.9,-0.7) ;
\coordinate (b) at (2,0) ;
\coordinate (x) at (0.9,-2) ;
\draw (aa) node{$A\sqcup A$}
(a') node{$A'$}
(a) node{$A$}
(bb) node{$B\sqcup B$}
(b') node{$B'$}
(b) node{$B$}
(x) node{$X$} ;
\draw[->] ($(aa)+(R)+(R)$) --($(a)+(L)$) node[midway, above]{$\nabla$} ;
\draw[->] ($(aa)+(D)$) --($(bb)+(U)$) node[midway, left]{$h\sqcup h$} ;
\draw[->] ($(bb)+(R)+(R)$) --($(b)+(L)$) node[pos=.6, above]{$\nabla$} ;
\draw[->] ($(a)+(D)$) --($(b)+(U)$) node[midway, right]{$h$} ;
\draw[htpcof, thick] ($(a')+(-.1,.1)$) --($(aa)+(D)+(R)$) ;
\draw[->] ($(bb)+(.2,-.2)$) --($(b')+(-.2,.2)$) ;
\draw[->>] ($(a')+(.2,.2)$) --($(a)+(D)+(L)$) node[sloped, pos=0.4, above]{$\sim$} ;
\draw[->>] ($(b')+(.2,.2)$) --($(b)+(D)+(L)$) node[sloped, pos=0.4, above]{$\sim$} ;
\draw[->] ($(b')+(D)$) --($(x)+(U)$) node[midway, right]{$H$} ;
\draw ($(bb)+(D)$) edge[->, out=-90, in=160] node[midway, left]{$[f\,g]$} ($(x)+(-.2,.2)$) ;
\draw[->, dotted] ($(a')+(D)$) --($(b')+(U)$) node[pos=.3, left]{$\al$} ;
\begin{scope}[xshift=5cm]
\coordinate (U) at (0,0.3) ;
\coordinate (D) at (0,-0.3) ;
\coordinate (L) at (-0.3,0) ;
\coordinate (R) at (0.3,0) ;
\coordinate (aa) at (0,2) ;
\coordinate (a') at (0,0) ;
\coordinate (b') at (2,2) ;
\coordinate (b) at (2,0) ;
\draw (aa) node{$A\sqcup A$} (a') node{$A'$} (b') node{$B'$} (b) node{$B$} ;
\draw[->] ($(aa)+(R)+(R)$) --($(b')+(L)$) ;
\draw[htpcof] ($(a')+(U)$) --($(aa)+(D)$) ;
\draw[->] ($(a')+(R)$) --($(b)+(L)$) ;
\draw[->>] ($(b')+(D)$) --($(b)+(U)$) node[midway, right]{$\wr$} ;
\draw[->, dotted] ($(a')+(U)+(R)$) --($(b')+(L)+(D)$) node[midway, left]{$\al$} node[midway, below right]{$\sim^l$} ;
\end{scope}
\end{tikzpicture}
\end{center}
Thanks to the lifting property of homotopy cofibrations, there is a morphism $\al$ such that, in the
diagram on the right-hand side, the upper triangle commutes and the two compositions in the lower triangle
are left homotopic. As can be seen from the diagram on the left-hand side,
the composition $H\al$ is then a homotopy from $fh$ to $gh$.
\end{proof}

\begin{prop}
Let $X,B$ be objects of $\cat$. If $X$ is fibrant then right homotopy is an equivalence relation
on $\cat(B,X)$.
\end{prop}

\begin{proof}
The proof is the same as in the case of model categories, thanks to the second part of Axiom 0.
Right homotopy is reflexive and symmetric even if $X$ is not fibrant. Assume $X$ fibrant.
Let us show that right homotopy is transitive on $\cat(B,X)$.
The main idea is dual to the following: If $H$, resp. $H'$, is a left homotopy
from $f$ to $g$, resp. from $g$ to $h$, one can glue (here glueing means push-out) the associated cylinder objects,
as in the picture below. This yields a new cylinder object.
Since $H$ and $H'$ coincide on the parts of the cylinders which are glued together, they define
a morphism on this cylinder object. We thus obtain a left homotopy from $f$ to $h$.
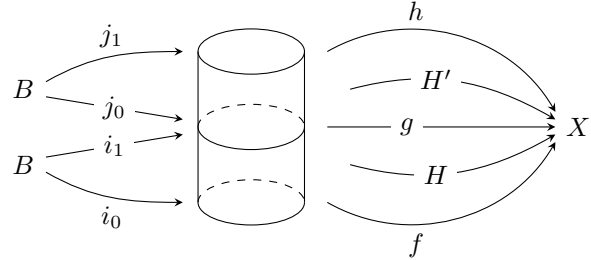
\begin{center}
\begin{tikzpicture}
\draw (3,0) circle (.7 and .3) ;
\draw (2.3,-1) arc (-180:0:.7 and .3) ;
\draw[dashed] (2.3,-1) arc (180:0:.7 and .3) ;
\draw (2.3,-2) arc (-180:0:.7 and .3) ;
\draw[dashed] (2.3,-2) arc (180:0:.7 and .3) ;
\draw (2.3,0) --(2.3,-2) ;
\draw (3.7,0) --(3.7,-2) ;
\draw (4,0) edge[->, out=30, in=120] node[pos=.35, above]{$h$} (7,-.8) ;
\draw (4.3,-.5) edge[->,out=15, in=150] node[pos=.4, fill=white]{$H'$} (7,-.9) ;
\draw (4,-1) edge[->] node[pos=.35, fill=white]{$g$} (7,-1) ;
\draw (4.3,-1.5) edge[->,out=-15, in=210] node[pos=.4, fill=white]{$H$} (7,-1.1) ;
\draw (4,-2) edge[->,out=-30, in=240] node[pos=.35, below]{$f$} (7,-1.2) ;
\draw (7.3,-1) node{$X$} ;
\draw (0,-.5) node{$B$} ;
\draw (0,-1.5) node{$B$} ;
\draw (.3,-.4) edge[->, out=30, in=180, above] node{$j_1$} (2.1,0) ;
\draw (.3,-.6) edge[->] node[fill=white]{$j_0$} (2.1,-.9) ;
\draw (.3,-1.4) edge[->] node[fill=white]{$i_1$} (2.1,-1.1) ;
\draw (.3,-1.6) edge[->, out=-30, in=180, below] node{$i_0$} (2.1,-2) ;
\end{tikzpicture}
\end{center}

Given a commutative diagram $(\ast)$,
let us show that $f$ and $g$ are right homotopic.

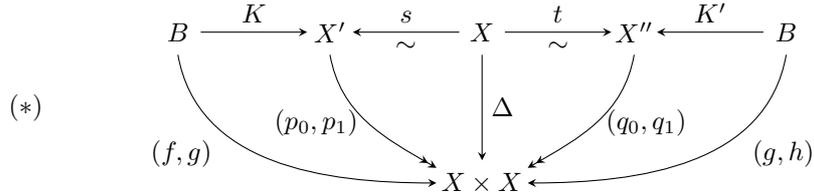
\begin{center}
\begin{tikzpicture}
\coordinate (x) at (0,0) ;
\coordinate (x') at (-2,0) ;
\coordinate (b) at (-4,0) ;
\coordinate (x'') at (2,0) ;
\coordinate (B) at (4,0) ;
\coordinate (xx) at (0,-2) ;
\coordinate (U) at (0,0.3) ;
\coordinate (D) at (0,-0.3) ;
\coordinate (L) at (-0.3,0) ;
\coordinate (R) at (0.3,0) ;
\draw (x) node{$X$} (x') node{$X'$} (b) node{$B$} (x'') node{$X''$}
(B) node{$B$} (xx) node{$X\times X$} ;
\draw[->] ($(b)+(R)$) --($(x')+(L)$) node[midway, above]{$K$} ;
\draw[->] ($(x)+(L)$) --($(x')+(R)$) node[midway, above]{$s$} node[midway, below]{$\sim$} ;
\draw[->] ($(x)+(R)$) --($(x'')+(L)$) node[midway, above]{$t$} node[midway, below]{$\sim$} ;
\draw[->] ($(B)+(L)$) --($(x'')+(R)$) node[midway, above]{$K'$} ;
\draw[->] ($(x)+(D)$) --($(xx)+(U)$) node[midway, right]{$\Delta$} ;
\draw ($(x')+(D)$) edge[->>, out=-80, in=150] node[midway, left]{$(p_0,p_1)$} ($(xx)+(-.6,.2)$) ;
\draw ($(x'')+(D)$) edge[->>, out=-100, in=30] node[midway, right]{$(q_0,q_1)$} ($(xx)+(.6,.2)$) ;
\draw ($(b)+(D)$) edge[->, out=-80, in=180] node[pos=.3, below left]{$(f,g)$} ($(xx)+(-.6,0)$) ;
\draw ($(B)+(D)$) edge[->, out=-100, in=0] node[pos=.3, below right]{$(g,h)$} ($(xx)+(.6,0)$) ;
\draw (-6,-1) node{$(\ast)$} ;
\end{tikzpicture}
\end{center}

We first note that $p_1$ is a trivial fibration:
Let $\pi_1$ be the second projection $X\times X\fl X$. We have
$p_1 s = \pi_1 (p_0,p_1) s = \pi_1 \Delta = 1$ so that, by the two-out-of-three property, $p_1$ is
a weak equivalence. Moreover, $(p_0,p_1)$ is a fibration, so that, in order to prove
that $p_1$ is a fibration, it is enough to prove that $\pi_1$ is a fibration.
The following diagram:
\[
\xymatrix{
X\times X \dr \bas & X \bas \\ X \dr & \ast
}
\]
is a pull-back diagram. By assumption $X$ is fibrant, so that
Lemma~\ref{lemma: first ppties}\,(d) applies and $\pi_1$ is
a fibration.

We now glue the paths objects $X', X''$ and the homotopies $K,K'$ together.
Let $(Y,a,b)$ be a pull-back of $q_0$ along $p_1$. Such a pull-back exists
by Axiom 0 since $p_1$ is a trivial fibrations and $q_0$ is an epimorphism.
In the diagam on the left-hand side:
\begin{center}
\begin{tikzpicture}[scale=.8]
\draw (-1,1) node{$X$} (0,0) node{$Y$} (2,0) node{$X''$} (0,-2) node{$X'$} (2,-2) node{$X$} ;
\draw[->, dotted] (-.8,.8) --(-.2,.2) node[pos=.3, right]{$u$} ;
\draw (-.7,1) edge[->, out=0, in=150] node[near start, above]{$t$} node[sloped, midway, above]{$\sim$} (1.8,0.2) ;
\draw (-1,.7) edge[->, out=-90, in=120] node[near start, left]{$s$} node[sloped, midway, below]{$\sim$} (-.2,-1.8) ;
\draw[->>] (.3,0) --(1.7,0) node[midway, below]{$\sim$} node[pos=.4, above]{$b$} ;
\draw[->>] (0,-.3) --(0,-1.7) node[midway, right]{$\wr$} node[pos=.4, left]{$a$} ;
\draw[->>] (2,-.3) --(2,-1.7) node[midway, left]{$\wr$} node[midway, right]{$q_0$} ;
\draw[->>] (.3,-2) --(1.7,-2) node[midway, above]{$\sim$} node[midway, below]{$p_1$} ;
\begin{scope}[xshift=5cm]
\draw (-1,1) node{$B$} (0,0) node{$Y$} (2,0) node{$X''$} (0,-2) node{$X'$} (2,-2) node{$X$} ;
\draw[->, dotted] (-.8,.8) --(-.2,.2) node[pos=.3, right]{$k$} ;
\draw (-.7,1) edge[->, out=0, in=150] node[near start, above]{$K'$} node[sloped, midway, above]{$\sim$} (1.8,0.2) ;
\draw (-1,.7) edge[->, out=-90, in=120] node[near start, left]{$K$} node[sloped, midway, below]{$\sim$} (-.2,-1.8) ;
\draw[->] (.3,0) --(1.7,0) node[pos=.4, above]{$b$} ;
\draw[->] (0,-.3) --(0,-1.7) node[pos=.4, left]{$a$} ;
\draw[->] (2,-.3) --(2,-1.7) node[midway, right]{$q_0$} ;
\draw[->] (.3,-2) --(1.7,-2) node[midway, below]{$p_1$} ;
\end{scope}
\end{tikzpicture}
\end{center}
we have $q_0 t = 1 = p_1 s$ so that there is a morphism $u$ with $au=s$ and $bu =t$.
By Axiom (0), $b$ is a trivial fibration.
By the two-out-of-three property, $u$ is a weak-equivalence, so that
the factorisation $X\stackrel{u}{\gfl}Y\gfl X\times X$, where the second morphism is $(p_0a, q_1b)$, is a path object
for $X$. In the diagram on the right-hand side, we have
$p_1K=g=q_0K'$ and there is a morphism $k$ with $ak= K$ and $bk=K'$.
The morphism $k$ is a right homotopy from $f$ to $h$: We have
$p_0ak = p_0K = f$ and $q_1bk = q_1K' = h$.
\end{proof}

We recall that $\cat_c$ is the full subcategory of $\cat$ whose objects
are the cofibrant objects.

\begin{lemma}[Ken Brown's lemma]\label{lemma: Ken Brown}
Let $F$ be a contravariant functor from $\cat_c$ to any category $\dc$, and let $\wc'$
be a class of morphisms in $\dc$, containing the identities, which satisfies:
for any composable $u,v$ in $\dc$, if $u\in\wc'$,
then: $uv\in\wc'$ is equivalent to $v\in\wc'$.
Assume that $F$ takes trivial homotopy cofibrations (between cofibrant objects) to $W'$.
Then $F$ takes all weak equivalences (between cofibrant objects) to $W'$.
\end{lemma}

\begin{proof}
Let $A,B$ be cofibrant objects in $\cat$, and let $A\stackrel{f}{\fl}B$ be a weak equivalence.
By Axiom 0, the morphisms $A\stackrel{i}{\fl}A\coprod B$ and  $B\stackrel{j}{\fl}A\coprod B$
are cofibrations, and $A\coprod B$ is cofibrant.
We can thus factor $A\coprod B \stackrel{[f\,1]}{\gfl}B$
as a homotopy cofibration $q$ followed by a weak fibration $p$:
\begin{center}
\begin{tikzpicture}
\coordinate (U) at (0,0.3) ;
\coordinate (D) at (0,-0.3) ;
\coordinate (L) at (-0.3,0) ;
\coordinate (R) at (0.3,0) ;
\coordinate (a) at (0,0) ;
\coordinate (ab) at (2,0) ;
\coordinate (b) at (4,0) ;
\coordinate (c) at (3,-1.5) ;
\draw (a) node{$A$} (ab) node{$A\sqcup B$} (b) node{$B$} (c) node{$C$} ;
\draw[>->] ($(a)+(R)$) --($(ab)+(L)+(L)$) node[midway, above]{$i$} ;
\draw[->] ($(ab)+(R)+(R)$) --($(b)+(L)$) node[midway, above]{$[f\;1]$} ;
\draw[<-open triangle 60 reversed] ($(c)+(-.2,.2)$) --($(ab)+(.2,-.3)$) node[midway, left]{$q$} ;
\draw[->] ($(c)+(.2,.2)$) --($(b)+(-.2,-.2)$) node[sloped, midway, above]{$\sim$} node[near start, right]{$p$} ;
\draw ($(c)+(L)$) edge[<-open triangle 60 reversed, out=180, in=-45] node[pos=.6, below left]{$qi$} ($(a)+(R)+(D)$) ;
\draw ($(a)+(U)+(R)$) edge[->, out=45, in=135] node[sloped, midway, above]{$\sim$} node[near start, above]{$f$} ($(b)+(L)+(U)$) ;
\begin{scope}[xshift=5cm]
\coordinate (U) at (0,0.3) ;
\coordinate (D) at (0,-0.3) ;
\coordinate (L) at (-0.3,0) ;
\coordinate (R) at (0.3,0) ;
\coordinate (a) at (0,0) ;
\coordinate (ab) at (2,0) ;
\coordinate (b) at (4,0) ;
\coordinate (c) at (3,-1.5) ;
\draw (a) node{$B$} (ab) node{$A\sqcup B$} (b) node{$B$} (c) node{$C$} ;
\draw[>->] ($(a)+(R)$) --($(ab)+(L)+(L)$) node[midway, above]{$j$} ;
\draw[->] ($(ab)+(R)+(R)$) --($(b)+(L)$) node[midway, above]{$[f\;1]$} ;
\draw[<-open triangle 60 reversed] ($(c)+(-.2,.2)$) --($(ab)+(.2,-.3)$) node[midway, left]{$q$} ;
\draw[->] ($(c)+(.2,.2)$) --($(b)+(-.2,-.2)$) node[sloped, midway, above]{$\sim$} node[near start, right]{$p$} ;
\draw ($(c)+(L)$) edge[<-, out=180, in=-45] node[pos=.6, below left]{$qj$} ($(a)+(R)+(D)$) ;
\draw ($(a)+(U)+(R)$) edge[->, out=45, in=135] node[midway, above]{$1$} ($(b)+(L)+(U)$) ;
\end{scope}
\end{tikzpicture}
\end{center}
We note that $C$ is cofibrant by Axiom 0.5.
By Lemma~\ref{lemma: composition htpcof}, $qi$ is a homotopy cofibration, and by the two-out-of-three
property, $qi$ is a weak equivalence, so that $F(qi)$ is in $\wc'$. The same argument shows that $F(qj)$
is in $\wc'$. We thus have $Fp\in\wc'$ and $Ff\in\wc'$.
\end{proof}

The dual version of Ken Brown's lemma will also be used. We state it without proof
since the proof is similar to the one above and exactly the same as for model categories.
We recall that $\cat_f$ is the full subcategory of $\cat$ whose objects are the fibrant objects.

\begin{lemma}[Ken Brown's lemma]\label{lemma: Ken Brown fibrant}
Let $F$ be a covariant functor from $\cat_f$ to any category $\dc$, and let $\wc'$
be a class of morphisms in $\dc$, containing the identities, which satisfies:
for any composable $u,v$ in $\dc$, if $u\in\wc'$,
then: $uv\in\wc'$ is equivalent to $v\in\wc'$.
Assume that $F$ takes trivial fibrations (between fibrant objects) to $W'$.
Then $F$ takes all weak equivalences (between fibrant objects) to $W'$.
\end{lemma}

\begin{prop}\label{proposition: h ast bijective}
Let $X$ be fibrant, $A,B$ be cofibrant and let $A\stackrel{h}{\gfl}B$ be a weak equivalence.
Then the induced map $\cat(B,X)/\!\sim^r\, \stackrel{h^\ast}{\gfl}\cat(A,X)/\!\sim^r$ is bijective.
\end{prop}

\begin{proof}
The map $h^\ast$ is well-defined since right homotopy is compatible with pre-compositions.
We first prove the statement when $h$ has the left lifting properties $\hsq$ and $\bsq$
with respect to all fibrations. The general case is then shown to be a consequence
of this first case, thanks to Ken Brown's lemma (Lemma~\ref{lemma: Ken Brown}).

Let us first assume that $h\hsq\Fib$ and $h\bsq\Fib$. The map $h^\ast$
is easily seen to be surjective: Let $A\stackrel{a}{\fl}X$ be a morphism in $\cat$. Since $X$ is fibrant,
there is a lift $B\stackrel{b}{\fl}X$ with $bh=a$ (and the composition $B\stackrel{b}{\fl} X \fl \ast$
is left homotopic to $B\fl \ast$).
The map $h^\ast$ is injective: Let $f,g$ be two morphisms from $B$ to $X$ and assume
that $fh$ and $gh$ are right homotopic. Then, there is a cylinder object
$X\stackrel{\sim}{\gfl}X' \stackrel{p}{\fib} X\times X$ and a right homotopy
$A\stackrel{K}{\gfl}X'$ from $fh$ to $gh$. There is a commutative square:
\begin{center}
\begin{tikzpicture}
\coordinate (a) at (0,2) ;
\coordinate (b) at (0,0) ;
\coordinate (x') at (2,2) ;
\coordinate (xx) at (2,0) ;
\coordinate (U) at (0,0.3) ;
\coordinate (D) at (0,-0.3) ;
\coordinate (L) at (-0.3,0) ;
\coordinate (R) at (0.3,0) ;
\draw (a) node{$A$} (b) node{$B$} (x') node{$X'$} (xx) node{$X\times X$} ;
\draw[->] ($(a)+(D)$) --($(b)+(U)$) node[midway, left]{$h$} ;
\draw[->] ($(a)+(R)$) --($(x')+(L)$) node[midway, above]{$K$} ;
\draw[->>] ($(x')+(D)$) --($(xx)+(U)$) node[midway, right]{$p$} ;
\draw[->] ($(b)+(R)$) --($(xx)+(L)+(L)$) node[midway, below]{$(f,g)$} ;
\draw[->,dotted] ($(b)+(U)+(R)$) --($(x')+(D)+(L)$) node[midway, below]{$\;\overline{K}$} node[pos=.4, above =8pt]{$\sim^r$} ;
\end{tikzpicture}
\end{center}
By assumption on $h$, there is a morphism $\overline{K}$ with $p\overline{K} =(f,g)$
(and $\overline{K}h \sim^r K$). This morphism $\overline{K}$ is a right homotopy from $f$ to $g$.

By Proposition~\ref{proposition: whtpcof implies htpwcof}, the statement thus holds when $h$
is a trivial homotopy cofibration .
In order to conclude, one can apply Ken Brown's lemma (Lemma~\ref{lemma: Ken Brown})
to the functor $\cat(-,X)/\!\sim^r$, from $\cat_c$ to the category of sets,
with $\wc'$ the class of bijections.
\end{proof}

\begin{lemma}\label{lemma: lhtp iff rhtp}
 Let $f,g : B\fl X$ be two morphisms.
\begin{itemize}
 \item[(i)] Assume that $X$ is fibrant and that $f\sim^r g$. Then $g\sim^l f$ for any choice
of a cylinder object $B\coprod B \htpcof B' \stackrel{\sim}{\fl} B$ (if such a cylinder object exists).
 \item[(ii)] Assume that $B$ is cofibrant and that $f\sim^l g$. Then $g\sim^r f$ for any choice of a path object
$X \stackrel{\sim}{\fl} X' \fib X\times X$.
\end{itemize}
\end{lemma}

\begin{proof}
(i) Assume that $X$ is fibrant and that there is a homotopy:
\[
\xymatrix{
X \dr^{\sim}_s \bdr_{\Delta} & X' \base^{^{^{\!(p_0,p_1)}}} & B \gau_K \bg^{(f,g)} \\
& X\times X &
}
\]
and let $B\coprod B \htpcof B' \stackrel{\sim}{\fl} B$ be a cylinder object for $B$.
By Lemma~\ref{lemma: first ppties}\,(d), the first projection $X\times X\fl X$ is a fibration,
so that the morphism $X'\stackrel{p_0}{\gfl}X$ is a fibration.
Since $p_0 s = 1$, it is also a weak equivalence. We can thus apply the lifting property of homotopy
cofibrations to the commutative square:
\begin{center}
\begin{tikzpicture}
\coordinate (U) at (0,0.3) ;
\coordinate (D) at (0,-0.3) ;
\coordinate (L) at (-0.3,0) ;
\coordinate (R) at (0.3,0) ;
\coordinate (bb) at (0,2) ;
\coordinate (b') at (0,0) ;
\coordinate (x') at (2,2) ;
\coordinate (x) at (2,0) ;
\draw (bb) node{$B\sqcup B$} (x') node{$X'$} (b') node{$B'$} (x) node{$X$} ;
\draw[->] ($(bb)+(R)+(R)$) --($(x')+(L)$) node[midway, above]{$[K\;sf]$} ;
\draw[htpcof] ($(b')+(U)$) --($(bb)+(D)$) node[midway, left]{$[i_0\;i_1]$} ;
\draw[->] ($(b')+(R)$) --($(x)+(L)$) node[midway, below]{$ft$} ;
\draw[->>] ($(x')+(D)$) --($(x)+(U)$) node[midway, right]{$\wr\;\;p_0$} ;
\draw[->, dotted] ($(b')+(U)+(R)$) --($(x')+(L)+(D)$) node[midway, left]{$\widetilde{K}$} node[midway, below right]{$\sim^l$} ;
\end{tikzpicture}
\end{center}
in order to construct a morphism $B'\stackrel{\widetilde{K}}{\gfl}X'$ such that $\widetilde{K}[i_0\;i_1] = [K\;sf]$.
The composition $p_1\widetilde{K}$ is then a left homotopy from $g$ to $f$, as the following
equalities show:
$p_1\widetilde{K}[i_0\;i_1] = p_1[K\;sf] = [g\;f]$.

(ii) Assume that $B$ is cofibrant and that $f$ is left homotopic to $g$.
By Lemma~\ref{lemma: specific homotopies}, there is a commutative diagram:
\begin{center}
\begin{tikzpicture}
\coordinate (bb) at (0,2) ;
\coordinate (x) at (-2,0) ;
\coordinate (b') at (0,0) ;
\coordinate (b) at (2,0) ;
\coordinate (U) at (0,0.3) ;
\coordinate (D) at (0,-0.3) ;
\coordinate (L) at (-0.3,0) ;
\coordinate (R) at (0.3,0) ;
\draw (bb) node{$B\sqcup B$} (x) node{$X$} (b') node{$B'$} (b) node{$B$} ;
\draw[->] ($(bb)+(D)+(L)$) --($(x)+(U)+(R)$) node[pos=.6, above left]{$[f\;g]$} ;
\draw[htpcof] ($(b')+(U)$) --($(bb)+(D)$) node[pos=.35, right]{$[i_0\;i_1]$} ;
\draw[->] ($(bb)+(D)+(R)$) --($(b)+(U)+(L)$) node[midway, above right]{$\nabla$} ;
\draw[->] ($(b')+(L)$) --($(x)+(R)$) node[midway, below]{$H$} ;
\draw[->] ($(b')+(R)$) --($(b)+(L)$) node[midway, above]{$\sim$} node[midway, below]{$s$} ;
\end{tikzpicture}
\end{center}
By Lemma~\ref{lemma: i0 htpcof}, the morphism $i_0$ is both a weak equivalence and a homotopy cofibration.
By Proposition~\ref{proposition: whtpcof implies htpwcof}, we have: $i_0\hsq\Fib$.
Let $X \stackrel{t\;\sim}{\gfl}X'\stackrel{(p_0,p_1)}{\fib}X\times X$ be a path object for $X$.
The following square commutes
\begin{center}
\begin{tikzpicture}
\coordinate (U) at (0,0.3) ;
\coordinate (D) at (0,-0.3) ;
\coordinate (L) at (-0.3,0) ;
\coordinate (R) at (0.3,0) ;
\coordinate (b) at (0,2) ;
\coordinate (b') at (0,0) ;
\coordinate (x') at (2,2) ;
\coordinate (xx) at (2,0) ;
\draw (b) node{$B$} (x') node{$X'$} (b') node{$B'$} (xx) node{$X\times X$} ;
\draw[->] ($(b)+(R)$) --($(x')+(L)$) node[midway, above]{$tf$} ;
\draw[htpcof] ($(b')+(U)$) --($(b)+(D)$) node[midway, left]{$i_0$} node[midway, right]{$\wr$} ;
\draw[->] ($(b')+(R)$) --($(xx)+(L)+(L)$) node[midway, below]{$(H,fs)$} ;
\draw[->>] ($(x')+(D)$) --($(xx)+(U)$) node[midway, right]{$(p_0,p_1)$} ;
\draw[->, dotted] ($(b')+(U)+(R)$) --($(x')+(L)+(D)$) node[midway, right]{$\widetilde{H}$} node[midway, above left]{$\sim^r$} ;
\end{tikzpicture}
\end{center}
so that there is a morphism $\widetilde{H}$ with $(p_0,p_1)\widetilde{H} = (H, fs)$ and
$\widetilde{H}i_0 \sim^r tf$.
The equalities below show that $\widetilde{H}i_1$ is a right homotopy from $g$ to $f$:
$(p_0,p_1)\widetilde{H}i_1 = (H,fs)i_1 = (g,f)$.
\end{proof}

\begin{cor}\label{corollary: homotopies coincide}
Let $X$ be fibrant and $B$ be cofibrant. Then left homotopy and right homotopy coincide
on $\cat(B,X)$. In particular left homotopy is an equivalence relation on $\cat(B,X)$.
\end{cor}

\begin{proof}
If $B$ is cofibrant, then there always exists a cylinder object of the form
$B\coprod B \htpcof B' \stackrel{\sim}{\fl} B$ by Axiom (4).
\end{proof}

\begin{prop}
Let $X,Y$ be fibrant, $B$ be cofibrant, and $h:X\stackrel{\sim}{\fl}Y$ be a weak equivalence.
Then the induced map $\cat(B,X)/\!\sim\; \stackrel{h_\ast}{\gfl}\cat(B,Y)/\!\sim$ is bijective.
\end{prop}

\begin{proof}
We note that the map $h_\ast$ is well-defined.
By Ken Brown's lemma, it is enough to consider the case where $h$ is a trivial fibration.
We thus assume that $h$ is a trivial fibration.
The map $h_\ast$ is surjective since $B$ is cofibrant and $h$ is a trivial fibration.
The map $h_\ast$ is injective: Let $f,g$ be two morphisms in $\cat(B,X)$.
Assume that $hf\sim hg$. Then, by Lemma~\ref{lemma: specific homotopies}\,(2),
there is a commutative diagram as on the left-hand side of:
\begin{center}
\begin{tikzpicture}
\coordinate (bb) at (0,2) ;
\coordinate (y) at (-2,0) ;
\coordinate (b') at (0,0) ;
\coordinate (b) at (2,0) ;
\coordinate (U) at (0,0.3) ;
\coordinate (D) at (0,-0.3) ;
\coordinate (L) at (-0.3,0) ;
\coordinate (R) at (0.3,0) ;
\draw (bb) node{$B\sqcup B$} (y) node{$Y$} (b') node{$B'$} (b) node{$B$} ;
\draw[->] ($(bb)+(D)+(L)$) --($(y)+(U)+(R)$) node[pos=.6, above left]{$[hf\;hg]$} ;
\draw[htpcof] ($(b')+(U)$) --($(bb)+(D)$) node[pos=.35, right]{$[i_0\;i_1]$} ;
\draw[->] ($(bb)+(D)+(R)$) --($(b)+(U)+(L)$) node[midway, above right]{$\nabla$} ;
\draw[->] ($(b')+(L)$) --($(y)+(R)$) node[midway, below]{$H$} ;
\draw[->] ($(b')+(R)$) --($(b)+(L)$) node[midway, above]{$\sim$} node[midway, below]{$s$} ;
\begin{scope}[xshift=5cm]
\coordinate (U) at (0,0.3) ;
\coordinate (D) at (0,-0.3) ;
\coordinate (L) at (-0.3,0) ;
\coordinate (R) at (0.3,0) ;
\coordinate (bb) at (0,2) ;
\coordinate (b') at (0,0) ;
\coordinate (x) at (2,2) ;
\coordinate (y) at (2,0) ;
\draw (bb) node{$B\sqcup B$} (x) node{$X$} (b') node{$B'$} (y) node{$Y.$} ;
\draw[->] ($(bb)+(R)+(R)$) --($(x)+(L)$) node[midway, above]{$[f\;g]$} ;
\draw[htpcof] ($(b')+(U)$) --($(bb)+(D)$) node[midway, left]{$[i_0\;i_1]$} ;
\draw[->] ($(b')+(R)$) --($(y)+(L)$) node[midway, below]{$H$} ;
\draw[->>] ($(x)+(D)$) --($(y)+(U)$) node[midway, right]{$\wr\;\;h$} ;
\draw[->, dotted] ($(b')+(U)+(R)$) --($(x)+(L)+(D)$) node[midway, left]{$\overline{H}$} node[midway, below right]{$\sim^l$} ;
\end{scope}
\end{tikzpicture}
\end{center}
The square on the right-hand side commutes so that there is a morphism $\overline{H}$ with
$\overline{H}[i_0\;i_1] = [f\;g]$ and $h\overline{H} \sim^l H$. This shows that $f$ is left homotopic to $g$.
\end{proof}

\begin{lemma}\label{lemma: htp to weq implies weq}
 Let $B$ be cofibrant, $X$ be fibrant and $B\stackrel{f}{\fl}X$ be a morphism.
If $f$ is homotopic to a weak equivalence, then $f$ is a weak equivalence.
\end{lemma}

\begin{proof}
Let $B\coprod B \stackrel{[i_0\;i_1]}{\gfl} B' \stackrel{s}{\gfl} B$ be a cylinder object for $B$
(so that $s$ is a weak equivalence),
and let $B'\stackrel{H}{\gfl}X$ be a left homotopy from $f$ to some weak equivalence $w$.
Then $si_0=1 = si_1$ so that $i_0$ and $i_1$ are weak equivalences by the two-out-of-three property.
Since $Hi_1 = w$, the morphism $H$ is a weak equivalence, and the equality $Hi_0 =f$ implies that
$f$ is a weak equivalence.
\end{proof}

We recall that a morphism $f$ is a homotopy equivalence if there is some $g$ such that
both $fg$ and $gf$ are homotopic to the respective identities.

\begin{prop}\label{proposition: weq and htpeq}
Let $X,Y$ be cofibrant and fibrant, and let $X\stackrel{f}{\fl}Y$ be a morphism.
Then $f$ is a weak equivalence if and only if it is a homotopy equivalence.
\end{prop}

\begin{proof}
Assume first that $f$ is a weak equivalence. By Proposition~\ref{proposition: h ast bijective}
and Corollary~\ref{corollary: homotopies coincide} the map $\cat(Y,X)/\!\sim\;\stackrel{f^\ast}{\gfl}\cat(X,X)/\!\sim$
is surjective. In particular, there is a morphism $Y\stackrel{g}{\fl}X$ such that $gf\sim 1$.
This implies, by Lemma~\ref{lemma: htp and composition}, $fgf\sim f$.
By Proposition~\ref{proposition: h ast bijective} and Corollary~\ref{corollary: homotopies coincide},
the map $\cat(Y,Y)/\!\sim\;\stackrel{f^\ast}{\gfl}\cat(X,Y)/\!\sim$ is injective
so that $fg\sim 1$ and $g$ is homotopy inverse to $f$.

Conversely, assume that $f$ is a homotopy equivalence.
Factor $f$ as a trivial cofibration $g : X \stackrel{\sim}{\cof} Z$ followed by a fibration $p : Z \fib Y$.
Let $f'$ be a homotopy inverse for $f$. By the first part of the proof, the
weak equivalence $g$ has some homotopy inverse $g'$.
We have $p\sim pgg' \sim fg'$, so that $gf'p \sim 1$ and, by Lemma~\ref{lemma: htp to weq implies weq},
$gf'p$ is a weak equivalence. If $p$ were a retract of $gf'p$, then it would be a weak equivalence and so would $f$ be.
In the diagram:
\[
\xymatrix{
Z \dreg \bas_p & Z \dreg \bas_{gf'p} & Z \bas^p \\
Y \dr^{gf'} & Z \dr^{fg'} & Y,
}
\]
only the left-hand square commutes. The right-hand square only commutes up to homotopy and the
composition of the bottom row is only homotopic to the identity. The idea is thus to replace
$gf'$ by some homotopic morphism for which the equalities will hold on the nose.
Let $Y'\stackrel{\widetilde{H}}{\gfl}Y$ be a left homotopy from $ff'$ to 1,
where $Y\coprod Y \stackrel{[i_0\;i_1]}{\htpcof} Y' \stackrel{\sim}{\gfl}Y$ is a cylinder object for $Y$.
Since the square:
\begin{center}
\begin{tikzpicture}
\coordinate (U) at (0,0.3) ;
\coordinate (D) at (0,-0.3) ;
\coordinate (L) at (-0.3,0) ;
\coordinate (R) at (0.3,0) ;
\coordinate (y) at (0,2) ;
\coordinate (y') at (0,0) ;
\coordinate (z) at (2,2) ;
\coordinate (Y) at (2,0) ;
\draw (y) node{$Y$} (z) node{$Z$} (y') node{$Y'$} (Y) node{$Y$} ;
\draw[->] ($(y)+(R)$) --($(z)+(L)$) node[midway, above]{$gf'$} ;
\draw[htpcof] ($(y')+(U)$) --($(y)+(D)$) node[midway, left]{$i_0$} node[midway, right]{$\wr$} ;
\draw[->] ($(y')+(R)$) --($(Y)+(L)$) node[midway, below]{$\widetilde{H}$} ;
\draw[->>] ($(z)+(D)$) --($(Y)+(U)$) node[midway, right]{$p$} ;
\draw[->, dotted] ($(y')+(U)+(R)$) --($(z)+(L)+(D)$) node[midway, right]{$H$} node[midway, above left]{$\sim^r$} ;
\end{tikzpicture}
\end{center}
commutes, Proposition~\ref{proposition: whtpcof implies htpwcof} gives a morphism $H$ such that $pH=\widetilde{H}$ and
$Hi_0\sim^r gf'$. We have $Hi_1 \sim^l Hi_0 \sim^r gf'$, and
$Hi_1 p \sim gf'fg' \sim 1$. By Lemma~\ref{lemma: htp to weq implies weq}, $Hi_1p$
is a weak equivalence. Moreover, the diagram:
\[
\xymatrix{
Z \dreg \bas_p & Z \dreg \bas_{Hi_1p} & Z \bas^p \\
Y \dr_{Hi_1} & Z \dr_{p} & Y,
}
\]
commutes and we have $pHi_1 = \widetilde{H}i_1 = 1$. Therefore, $p$ is a retract of
a weak equivalence: It is a weak equivalence, and so is $f$.
\end{proof}

\subsection{Cofibrantly generated weak model structures}\label{ssection: cof gene}

When trying to find a model category structure, one usually knows what the weak equivalences are,
and what the cofibrant (or fibrant) objects should be. It seems quite difficult in general to
define some classes of fibrations and cofibrations. The usual trick is the following one:
In a model category, the fibrations are exactly those morphisms that have the right lifting property
with respect to all trivial cofibrations. Similarly, the trivial fibrations are those morphisms that have
the right lifting property with respect to all cofibrations. It turns out that it is often enough,
in order for a morphism to be a fibration (resp. trivial fibration), that it has the right lifting property with
respect to some subset $J$ (resp. $I$) of the trivial cofibrations (resp. cofibrations) only.
Lifting property with respect to all trivial cofibrations (resp. cofibrations) is then automatic.
Model category structures are thus often defined via such sets $I$, $J$. Morphisms in $I$ are
called generating cofibrations and morphisms in $J$ generating trivial cofibrations.

Given two sets $I$ and $J$, one can define a fibration to be a morphism in $J^\square$.
Morally the class of trivial fibrations should be $I^\square$, so that one defines
the class of cofibrations to be $\prescript{\square}{}{(I^\square)}$. A few properties
have to be satisfied in order to ensure these definitions to give rise to a model category structure
and the existence of factorisations then comes from the so-called small object argument. A model category
structure defined in that way is called \emph{cofibrantly generated}.

In the setup of Section~\ref{section: w-model str from rigids}, the strategy is similar to the one described above.
The main difference is that the small object argument cannot apply in a category without limits and colimits. The factorisations
are thus given explicitely, and this is the reason for introducing a weaker version of the factorisation axiom.

\begin{defi}\rm
A left-weak model category structure is called \emph{cofibratly generated} if there are two sets
of morphisms $I$ and $J$ such that:
$\Fib = J^\square$, $\wc\cap\Fib = I^\square$ and $\Cof = \prescript{\square}{}{(I^\square)}$.
Morphisms in $I$ (resp. $J$) are then called \emph{generating cofibrations} (resp. \emph{generating trivial cofibrations}).
\end{defi}

\begin{rk}
\begin{itemize}
 \item[(a)] This definition is a bit improper, since the sets $I$ and $J$ are not required to permit the small object argument.
 \item[(b)] Unlike the case of model category structures, the condition on cofibrations has to be added since it does not
follow from the axioms. 
\end{itemize}

\end{rk}

The following proposition is a much weaker version of a theorem of Dan Kan (see e.g. \cite[Theorem 2.1.19]{Hovey}).

\begin{prop}\label{proposition: cof gene}
Let $\cat$ be a category with finite products and coproducts.
Let $\wc$ be a subcategory of $\cat$, and let $I$ and $J$ be two sets of morphisms in $\cat$.
Let $\Fib$, $w\Fib$, $\Cof$ and $w\,\Cof$ be defined as follows:
$\Fib = J^\square$, $w\Fib = I^\square$, $\Cof = \prescript{\square}{}{w\Fib}$
and $w\,\Cof = \prescript{\square}{}{\Fib}$.
Then the classes $\wc$, $\Fib$, $\Cof$ define a left-weak model category structure on $\cat$,
which is cofibrantly generated by $I$ and $J$,
if and only if the following properties are satisfied:
\begin{enumerate}
 \item There exist pull-backs of morphisms in $w\Fib$ along epimorphisms;
 \item The class $\wc$ is closed under retracts and satisfies the two-out-of-three property;
 \item $w\,\Cof \subseteq \Cof\cap\wc$;
 \item $w\Fib = \Fib\cap\wc$;
 \item Any morphism can be factored out as a morphism in $w\,\Cof$ followed
 by a morphism in $\Fib$.
 \item For any object $A$ such that the canonical morphism $\emptyset\fl A$ belongs to
 $\Cof$, any morphism with domain $A$ can be factored out as a morphism in $\prescript{\hsq}{}{w\Fib}\cap\prescript{\bsq}{}{w\Fib}$
 followed by a morphism in $w\Fib$.
\end{enumerate}
\end{prop}

\begin{rk}
\begin{itemize}
 \item Note that, in order to define left homotopy and right homotopy, one only has to choose a subcategory of weak equivalences
so that condition (6) makes sense.
 \item If conditions (1) to (6) are satisfied, then we also have $w\,\Cof = \Cof\cap\wc$.
\end{itemize}
\end{rk}

\begin{proof}
Let $\cat$ have finite products and coproducts.
For any left-weak model category structure $(\wc,\Fib,\Cof)$ on $\cat$, the following properties are satisfied:
(1), (2), $\prescript{\square}{}{\Fib} = \Cof\cap\wc$, $\Fib\cap\wc\subseteq\Cof^\square$ and (5).
If moreover the left-weak model category structure is cofibrantly generated, then (4), and thus also (6), holds.
Let us check that conditions (1) through (6) imply that the classes $\wc,\Fib,\Cof$ defined in the statement of the theorem
define a left-weak model structure which is cofibrantly generated by $I$ and $J$.

\noindent \emph{Axioms} (0) and (1) are required to hold by assumption (note that $w\Fib$ is stable under pullbacks).

\noindent \emph{Axiom} (0.5): Since cofibrations are defined by some left lifting property,
they are stable under taking push-outs. The square
\[
\xymatrix{
\emptyset \dr \bas & B \bas \\
A \dr & A\sqcup B
}
\]
being a push-out, the inclusion $A\fl A\coprod B$ is a cofibration if $A$ is cofibrant.
Let $A$ be cofibrant and let $A\htpcof B$ be a homotopy cofibration. For any
trivial fibration $X \fl Y$ and any morphism $B\fl Y$, there is a diagram:
\begin{center}
\begin{tikzpicture}
\coordinate (U) at (0,0.3) ;
\coordinate (D) at (0,-0.3) ;
\coordinate (L) at (-0.3,0) ;
\coordinate (R) at (0.3,0) ; 
\coordinate (a) at (0,2) ;
\coordinate (b) at (0,0) ;
\coordinate (x) at (2,2) ;
\coordinate (y) at (2,0) ;
\draw (a) node{$A$} (x) node{$X$} (b) node{$B$} (y) node{$Y.$} ;
\draw[->, dashed] ($(a)+(R)$) --($(x)+(L)$) node[midway, above]{$b$} ;
\draw[htpcof] ($(b)+(U)$) --($(a)+(D)$) node[midway, left]{$i$} ;
\draw[->] ($(b)+(R)$) --($(y)+(L)$) node[midway, below]{$a$} ;
\draw[->>] ($(x)+(D)$) --($(y)+(U)$) node[midway, right]{$p$} node[midway, left]{$\wr$};
\draw[->, dotted] ($(b)+(U)+(R)$) --($(x)+(L)+(D)$) node[midway, right]{$\al$} node[midway, above left]{$\sim^r$} ;
\end{tikzpicture}
\end{center}
Indeed, $A$ is cofibrant so that there is a morphism $b$ such that $pb = ai$,
and $i$ is a homotopy cofibration, so that there is some $\al$ such that
$p\al = a$ and $\al i \sim^r b$. In particular, $B$ is projective relative to all trivial fibrations.
By assumption, $w\Fib \subseteq \Fib\cap\wc$, so that $B$ is cofibrant.

\noindent \emph{Axiom} (2): By assumption, the class $\wc$ is stable
under compositions and under retracts. This is also true for $\Cof$ and $\Fib$ since
any left-$\square$ or right-$\square$ is stable under compositions and retracts.

\noindent \emph{Axiom} (3):
We have $\Cof \subseteq \prescript{\square}{}{(\Fib\cap\wc)}$ since $J^\square\cap\wc\subseteq I^\square$.
Let $f$ be a morphism in $\wc\cap\Cof$ and consider a factorisation of $f$ as a morphism $i$ in $w\,\Cof$ followed
by a morphism $p$ in $\Fib$. We have $w\,\Cof\subseteq\wc$ so that, by the two-out-of-three property, the morphism $p$ is in $\wc$.
Since $Fib\cap\wc\subseteq w\Fib$ and $f\in\Cof$, $f$ has the left lifting property with respect to $p$. It follows that
$f$ is a retract of $i$, so that $f$ belongs to $\prescript{\square}{}{\Fib}$.

\noindent \emph{Axiom} (4) follows from conditions (5) and (6) since $w\,\Cof\subseteq \Cof\cap\wc$ and $w\Fib \subseteq \Fib\cap\wc$.
\end{proof}

\section{Left-weak model category structures associated with rigid subcategories}\label{section: w-model str from rigids}

In this section, we show that there is a left-weak model category structure
on any triangulated category with a given contravariantly finite, rigid subcategory.

\subsection{Statement of main result}

In this section, $\cat$ is a triangulated category with a
contravariantly finite, rigid subcategory $\tc$.
Without loss of generality, $\tc$ is also assumed strictly full, and
stable under taking direct summands.

We refer to the beginning of \cite{Happel} or to the first chapter of~\cite{HJR}
for soft introductions to the basics of triangulated categories.

\begin{rk}
Most of the proof of Theorem~\ref{theorem: main} would hold
for $\tc$ extension-closed (instead of rigid), by replacing the use of Lemma~\ref{lemma: ideal}
by the use of triangulated Wakamatsu's lemma (\cite[Lemma 2.1]{JorgensenWaka}).
However, in order to hold in such a generality, the main statements would have to be deeply modified.
Indeed, it is a key property here that the ideal $(\shift\tc)$ is contained in $(\tc^\perp)$
(see Lemma~\ref{lemma: J-cof} for instance).
\end{rk}

If $\dc$ is a subcategory of $\cat$, we write $(\dc)$
for the ideal of morphisms factoring through some object in $\dc$,
and $\dc^\perp$ for the full subcategory of $\cat$ whose objects $X$ satisfy
$\cat(-,X)|_\dc = 0$. A morphism $D\stackrel{\al}{\fl} X$ is called a right $\dc$-approximation
if $D\in\dc$ and any morphism $D'\fl X$ with $D'\in\dc$ factors through $\al$.
There is the dual notion of a left-approximation.
A subcategory $\dc$ is called rigid if
$\shift\dc\subseteq\dc^\perp$. It is called contravariantly finite if, for any object $X\in\cat$,
there is a right $\dc$-approximation $D\fl X$.
Since the subcategories $(\shift\tc)^\perp$ and $\shift(\tc^\perp)$
coincide, the brackets will be omitted.
We write $\wc$ for the class of all those morphisms $X\stackrel{f}{\gfl} Y$
such that, for any (equivalently: some) triangle
$Z \stackrel{u}{\gfl} X \stackrel{f}{\gfl} Y \stackrel{v}{\gfl} \shift Z$,
the morphisms $u$ and $v$ are in $(\tperp)$.

\begin{theo}\label{theorem: main}
Let $\cat$ be a triangulated category with a
contravariantly finite, rigid subcategory $\tc$.
Then there is a left-weak model category structure on $\cat$
(as in Definition~\ref{Definition: weak-model})
with weak equivalences $\wc$, null-homotopic morphisms
$(\tc^\perp)$ and fibrant-cofibrant objects $\tc\ast\shift\tc$.
Moreover, if $\tc$ is skeletally small, then this left-weak model category structure
is cofibrantly generated, in the weak sense of Section~\ref{ssection: cof gene}.
\end{theo}

More precisely:
\begin{itemize}
 \item The fibrations are all those morphisms whose cones
belong to the ideal $(\shift\tc^\perp)$.
 \item The weak cofibrations are all morphisms isomorphic to some
$X\stackrel{\left[^0_1\right]}{\gfl} X\oplus\shift\tc$, where
$X\in\cat$ and $T\in\tc$.
 \item All objects are fibrant.
 \item The full subcategory of cofibrant objects is $\ccf$.
 \item Two morphisms are right-homotopic if and only if their difference belongs to the ideal $(\tc^\perp)$.
 \item A factorisation of a morphism $f$ into a weak cofibration followed by a fibration
is given by $X\stackrel{\left[^1_0\right]}{\gfl} X\oplus\shift T
\stackrel{[f\,\al]}{\gfl} Y$, where $\al$ is a (minimal) right $\shift\tc$-approximation
of $Y$. This factorisation is only functorial in $\cat/(\tc^\perp)$.
 \item A factorisation into a homotopy cofibration followed by a trivial fibration is constructed in Lemma~\ref{lemma: second factorisation}.
 \item For any $X\in\cat$, there is a right $\ccf$-approximation $QX\fl X$ (see Lemma~\ref{lemma: key lemma}, due to Buan--Marsh).
This approximation is a trivial fibration, and can thus be viewed as a cofibrant replacement.
\end{itemize}

\begin{cor}\label{corollary: main}
 The inclusion of $\ccf$ into $\cat$ induces an equivalence of categories:
\[\tc\ast\shift\tc / (\shift\tc) \simeq \cat[\wc^{-1}].\]
\end{cor}

\subsection{Proof of Theorem~\ref{theorem: main}}\label{section: proof}

Let $I$ be the class of all morphisms of the form
$T\fl Y$, where $T\in\tc$ and $Y\in \tc\ast\shift\tc$.
Let $J$ be the class of all morphisms of the form
$0 \fl \shift T$, where $T\in\tc$.

\begin{rk}
It is well-known since \cite{BMRRT} that if $\tc$ is contravariantly finite and rigid,
then $\tc\ast(\tc^\perp) = \cat$. Indeed, if $X$ is any object in $\cat$,
the cone of a right $\tc$-approximation of $X$ belongs to $\tc^\perp$ (this also holds more generally
for any contravariantly finite, extension-closed $\tc$ by \cite[Lemma 2.1]{JorgensenWaka}).
We will use this result several times in this section without mentionning it explicitely.
\end{rk}

\begin{lemma}\label{lemma: ideal}\cite[Lemma 2.3]{BMloc1}
 Let $f$ be a morphism in $\cat$. Then the morphism
$\cat(-,f)|_\tc$ is zero if and only if $f$ belongs to the ideal $\ideal$.
\end{lemma}

\begin{proof}
The proof follows from the fact that every object of $\cat$, and in particular the domain of $f$,
belongs to $\tc\ast \tc^\perp$.
\end{proof}

\begin{lemma}\label{lemma: wc'}
Let $\rc,\sc$ be contravariantly finite, extension-closed, full subcategories of $\cat$.
Let $\wc'$ be the class of morphisms $w$ in $\cat$ such that, in any triangle
\[Z \stackrel{r}{\gfl} X \stackrel{w}{\gfl} Y \stackrel{s}{\gfl} \shift Z \]
the morphism $r$ belongs to $\rideal$ and $s$ to $\sideal$.
Let $f$, $g$ be two composable morphisms in $\cat$.
\begin{enumerate}
 \item If $f$ and $g$ belong to $\wc'$, then $gf$ belongs to $\wc'$.
 \item Assume that $\rc \subseteq \sc$. If $f$ and $gf$ belong to $\wc'$, then
$g$ belongs to $\wc'$.
 \item Assume that $\sc\subseteq\rc$. If $g$ and $gf$ belong to $\wc'$, then
$f$ belongs to $\wc'$.
\end{enumerate}
\end{lemma}

\begin{proof}
Let $f,g$ be composable morphisms in $\cat$.

Let us first assume that $f$ and $g$ belong to $\wc'$.
The ocaedron axiom gives a commutative diagram whose rows and columns are triangles:
\[
\xymatrix{
& & C \bas^{u\in\rideal} & \\
A \dr^{a\in\rideal} \bas & X\baseg \dr^f & Y \bas^g \dr^{b\in\sideal} & \shift A \bas \\
C \dr^\eta & X \dr^{g\circ f} & Z \dr^\eps \bas^{v\in\sideal} & \shift C \\
& & \shift B &
}
\]
We use lemma~\ref{lemma: ideal} in order to prove that $\eta$ belongs to $\rideal$
and that $\eps$ belongs to $\sideal$.
For all $R\in\rc$, $S\in\sc$ and all morphisms $h: R\fl Z$ and $k: S\fl C$, we have:
$vh = 0$ so that there is some $w: R\fl Y$ with $gw = h$. Moreover,
$w$ factors through $f$ since $bw = 0$. This shows that $h$ factors through $g\circ f$,
and thus $\eps h = 0$.
On the other hand, $gf\eta k = 0$ so that $f\eta k$ factors through $u$. This implies
$f\eta k =0$ so that $\eta k$ factors through $a$. We thus have $\eta f =0$.

We next assume that $\rc\subseteq\sc$ and that $f$ and $g\circ f$ are in $\wc'$.
Applying the octeadron axiom gives a commutative diagram of triangles:
\[
\xymatrix{
& & C \bas_u & \\
& X \baseg \dr^f & Y \bas_g \dr^{\sideal} & \shift A \bas \\
B \dr^{\rideal} & X \dr^{g\circ f} & Z \bas_v \dr^{\sideal} & \shift B \bas \\
& & \shift C \dreg & \shift C 
}
\]
The lower-right square being commutative, the morphism $v$ belongs to the ideal $\sideal$.
We use lemma~\ref{lemma: ideal} in order to prove that $u$ belongs to $\rideal$.
For any $R$ in $\rc$, and any morphism $h$ from $R$ to $C$, the composition $uh$ factors
through $f$ since the cone of $f$ belongs to $\sideal$ and $\rc\subseteq\sc$. There is some $w$ from $S$
to $X$ such that $fw = uh$. We then have $gfw = guh = 0$, so that $w$ factors through $\rideal$.
This implies $w=0$ and thus $uh = 0$.

Finally, let us assume that $\sc\subseteq\rc$ and that $g$ and $g\circ f$ belong to $\wc'$.
Let us thus consider the following commutative diagram
of triangles:
\[
\xymatrix{
A \baseg \dr & B \bas^{\rideal} \dr & C \bas^{\rideal} & \\
A \dr^a & X \bas_{g\circ f} \dr^f & Y \dr^b \bas_g & \shift A \\
& Z \dreg \bas_{\sideal} & Z & \\
& \shift B & &
}
\]
Commutativity of the top-left square shows that $a$ belongs to $\rideal$.
Let us apply Lemma~\ref{lemma: ideal} to the morphism $b$:
For any $S$ in $\sc$ and any morphism $h$ from $S$ to $Y$, the composition
$gh$ factors through $g\circ f$. There is some $c$ from $S$ to $X$ such that
$gh = gfc$. We have $g(h-fc) = 0$, which implies $h = fc$ (since $\sc\subseteq\rc$) and thus $bh = 0$.
\end{proof}

\begin{cor}\label{corollary: 2-out-of-3}
The set $\wc$ satisfies the 2-out-of-3 property.
\end{cor}

\begin{rk}\label{rk: 2-out-of-6}
The set $\wc$ satisfies the more general 2-out-of-6 property.
\end{rk}

\begin{proof}
Let $f,g,h$ be composable morphisms in $\cat$. Assume that both $gf$ and
$hg$ belong to $\wc$. Let us prove that $f,g,h$ (and thus $fgh$) belong to $\wc$.
Since $gf$ and $hg$ belong to $\wc$, the octahedron axiom (applied twice) implies that $g$ belongs to $\wc$.
The conclusion thus follows from Corollary~\ref{corollary: 2-out-of-3}.
\end{proof}

\begin{lemma}\label{lemma: retracts}
The set $\wc$ is stable under retracts.
\end{lemma}

\begin{proof}
Let $f$ be a retract of a morphism $w\in\wc$. Complete $f$ to a triangle
$(f,g,h)$ and $w$ to a triangle $(w,x,y)$. Then $x$ belongs to $\ideal$ and $g$ is a retract of $x$.
The result follows since $\ideal$ is stable under retracts.
\end{proof}

\begin{lemma}\label{lemma: J-inj}
The set $J^\square$ consists of all those morphisms whose cone
factors through $\shift\tperp$.
\end{lemma}

\begin{proof}
The proof is similar to that of lemma~\ref{lemma: ideal}.
Let $f$ be a morphism in $\cat$, whose cone we denote by $g$.
Then $f$ belongs to $J^\square$ if and only if the morphism $\cat(-,g)|_{\shift\tc}$
is zero, if and only if
$g$ belongs to $(\shift\tc^\perp)$.
\end{proof}

\begin{cor}\label{corollary: fibrant}
Every object is fibrant.
\end{cor}

\begin{lemma}\label{lemma: I-inj}
We have $I^\square = (J^\square)\cap\wc$.
\end{lemma}

\begin{proof}
Let $f:A\fl B$ be a morphism in $(J^\square)\cap\wc$, and let
$g:T \fl Y$ be a morphism in $\cat$ with $T\in\tc$ and
$Y\in\ccf$. Complete $f$ to a triangle
$\xymatrix@-.6pc{C \dr^u & A \dr^f & B \dr^v & \shift C}$.
We then have $u\in\ideal$ and $v\in\ideal\cap(\shift\tc^\perp)$
by definition of $\wc$ and by Lemma~\ref{lemma: J-inj}.
Let 
\[\xymatrix{
T \bas_g \dr^a & A \bas^f \\
Y \dr^b & B 
}\]
be a commutative square.
Claim: the composition $vb$ vanishes. Indeed,
by assumption, there is a triangle $T_1\fl T_0\fl Y\fl \shift T_1$ with $T_0,T_1\in\tc$.
In the following diagram:
\[\xymatrix{
T_0 \bas_\al & \\
Y \dr^b \bas_\beta & B \bas^v \\
\shift T_1 \drp_\gamma \ar@/_.6pc/@{-->}[ur]^\delta & \shift C \bas^{-\shift u} \\
& \shift A,
}\]
the composition $vb\al$ vanishes since $v$ belongs to the ideal $(\tc^\perp)$.
Thus, there is a morphism $\gamma$ such that $\gamma\beta = vb$. Since $\shift u$ belongs to
$(\shift\tc^\perp)$, we have $(\shift u)\gamma = 0$ so that there is a morphism $\delta$
with $v\delta = \gamma$. Since $v\in (\shift\tc^\perp)$, we have $v\delta = 0$, so that $\gamma$
vanishes, and so does $vb$.

It follows from the claim above that, in the following diagram
\[\xymatrix{
& C \bas^u \\
T \bas_g \dr^a \hdrp^d & A \bas^f \\
Y \dr^b \hdrp^c & B \bas^v\\
& \shift C,
}\]
there is some $c:Y\fl A$ such that $b = fc$. We have
$f(cg - a) = 0$ so that there is some $d:T\fl C$ with
$a = cg + ud$. Since $u$ belongs to $\ideal$, we have $a = cg$.
This shows that $f$ belongs to $I^\square$.

Conversely, let $f:A\fl B$ be a morphism in $I^\square$.
For any $T\in\tc$, the morphism $f$ has the right lifting property with respect to
$0\fl T$ and to $0\fl \shift T$. This shows, by lemma~\ref{lemma: ideal}, that the cone
of $f$ belongs to $\ideal$ and to $(\shift\tc^\perp)$.
Let $\xymatrix@-.6pc{C \dr^u & A \dr^f & B \dr & \shift C}$ be a triangle in $\cat$.
Let $\al : T\fl C$ and $\beta: U \fl A$ be right $\tc$-approximations.
There is a commutative diagram whose columns are triangles:
\[\xymatrix{
T \dr^\al \bas_v & C \bas^u \\
U \bas_g \dr^\beta & A \bas^f \\
Y \dr \hdrp^\gamma & B
}\]
Since $f$ is in $J^\square$, it has the right lifting property with respect to
$g$, and there is some $\gamma: Y\fl A$ such that $\beta = \gamma g$.
It follows that $\beta v$ vanishes and thus that $u$ factors through the cone
of $\al$, which belongs to $\tc^\perp$.
\end{proof}

\begin{prop}\label{proposition: cofibrant}
Let $A$ be an object of $\cat$. Then $A$ is cofibrant if and only if
$A$ belongs to the subcategory $\ccf$.
\end{prop}

\begin{proof}
Let $A$ be in $\ccf$. Then the same proof as in Lemma~\ref{lemma: I-inj} shows that $A$ is cofibrant.
Conversely, assume that $A$ is a cofibrant object.
Let $T\fl A \fl Y \fl \shift T$ be a triangle in $\cat$ with
$T$ in $\tc$ and $Y$ in $\tperp$, and let
$Z \fl \shift U \fl Y \fl \shift Z$ be a triangle with $U\in\tc$ and $Z\in\tperp$.
In the following diagram:
\[\xymatrix{
& \shift U \bas^a & \\
A \dr^\al \hdrp^c & Y \bas_b \dr^\beta & \shift T \bgp^d \\
& \shift Z &
}\]
the morphism $a$ belongs to $I^\square$ by lemma~\ref{lemma: I-inj}. Thus
there is a morphism $c:A\fl\shift U$ such that $\al = ac$.
This implies that $b\al = 0$ and there is a morphism $d:\shift T \fl \shift Z$
such that $b = d\beta$. Since $Z\in\tperp$, we have $d=0$ so that
$b =0$ and $Y$ is a summand of $\shift U$. Therefore $A$ lies in $\ccf$.
\end{proof}

\begin{lemma}\label{lemma: J-cof}
A morphism is in $\prescript{\square}{}{(J^\square)}$ if and only if it is isomorphic to
$X\stackrel{\left[^1_0\right]}{\gfl} X\oplus \shift T$, for some
$X\in\cat$ and some $T\in\tc$.
\end{lemma}

\begin{proof}
Let $T\in\tc$. Since $0\fl\shift T$ is in $J$, it is in $\prescript{\square}{}{(J^\square)}$,
and so is any morphism isomorphic to some $X\fl X\oplus\shift T$.
Conversely, let $f:X\fl Y$ be a morphism in $\prescript{\square}{}{(J^\square)}$. By lemma~\ref{lemma: J-inj},
the morphism $X\fl 0$ is in $J^\square$, so that $f$ has the left lifting property
with respect to $X\fl 0$. This implies that $f$ is part of a split triangle and that
$f$ is isomorphic to $\xymatrix{X\dr^{\left[^0_1\right]} & X\oplus Y'}$,
for some $Y'\in\cat$.
Moreover, the morphism $f$ is in $\prescript{\square}{}{(J^\square)}$ if and only if
so is the morphism $0 \fl Y'$, if and only if $Y'$ belongs to $\prescript{\perp}{}{\shift\tc^\perp}$.
By \cite[Lemma 2.2]{BMloc1}, this last subcategory coincide with $\shift\tc$. 
\end{proof}

\begin{cor}\label{corollary: first factorisation}
Every morphism in $\cat$ factors as a morphism in $\prescript{\square}{}{(J^\square)}$ followed by a morphism in $J^\square$.
\end{cor}

\begin{proof}
Any morphism $X\stackrel{f}{ \fl}Y$ in $\cat$ factors as
$X\stackrel{\left[^1_0\right]}{\gfl} X\oplus\shift T
\stackrel{[f\,\al]}{\gfl} Y$, where $\al$ is a right $\shift\tc$-approximation
of $Y$. By Lemma~\ref{lemma: J-cof}, the first morphism is in $\prescript{\square}{}{(J^\square)}$.
Since $[f\,\al]$ is a right $\shift\tc$-approximation, its cone factors through $\shift\tc^\perp$.
Lemma~\ref{lemma: J-inj} shows that $[f\,\al]$ is in $J^\square$.
\end{proof}

\begin{cor}\label{cor: I-cof cap W}
We have $\prescript{\square}{}{(J^\square)} \subseteq \wc\cap\prescript{\square}{}{(I^\square)}$.
\end{cor}

\begin{proof}
Indeed, any morphism in $\prescript{\square}{}{(J^\square)}$ is isomorphic
to some $X\fl X\oplus\shift T$, $T\in\tc$,
and is thus a weak equivalence.
\end{proof}

\begin{lemma}\label{lemma: htp}
Let $X,Y\in\cat$ and let $f$ and $g$ be two morphisms in $\cat(X,Y)$. The following are equivalent:
\begin{itemize}
 \item[(i)] $f\sim^l g$;
 \item[(ii)] $f\sim^r g$;
 \item[(iii)] $f-g\in(\tc^\perp)$.
\end{itemize}
In particular, if $X$ belongs to $\ccf$, then $f\sim g$ is equivalent to $f-g\in(\shift\tc)$.
\end{lemma}

\begin{proof}
We fist note that a factorisation of $\nabla$ is a cylinder object for $X$ if and only if
it is isomorphic to some $X\oplus X \stackrel{\left[^{1\;1}_{a\;b}\right]}{\gfl} X\oplus U \stackrel{[1\;0]}{\gfl}X$,
with $U\in\tc^\perp$. Indeed, if $\nabla$ factors as $X\oplus X \stackrel{[\al\;\beta]}{\gfl}X' \stackrel{s}{\gfl}X$,
then $s\al = 1$ so that $s$ is isomorphic to some $X\oplus U \stackrel{[1\;0]}{\gfl}X$. The morphism $s$ is thus in $\wc$ if and only if
$U$ belongs to $\tc^\perp$.
Similarly, path objects for $Y$ are exactly those factorisations of $\Delta$ that are isomorphic to some
$Y\stackrel{[1\;0]}{\gfl}Y\oplus V \stackrel{\left[^{1\;1}_{c\;d}\right]}{\gfl}Y\oplus Y$, with $V\in\tc^\perp$.

Assume first that $f-g\in\cat(X,Y)$ factors through $\tc^\perp$. Let $T\fl X\stackrel{\al}{\fl} U\fl \shift T$
be a triangle in $\cat$ with $T\in\tc$ and $U\in\tc^\perp$. Since $f-g$ is in $(\tc^\perp)$,
there is a morphism $U\stackrel{h}{\fl}Y$ such that $f-g = h\al$. The following two commutative diagrams show that $f\sim^r g$
and $f\sim^l g$:
\[
\xymatrix{
 & X\oplus X \bg_{[f\;g]} \bas^{\left[^{1\;1}_{\al\;0}\right]} \bdr^{\nabla} & &
 & X \dr^{[g\;\al]} \bdr_{[f\;g]} & Y\oplus U \bas^{\left[^{1\;1}_{h\;0}\right]} & Y \bg^{\Delta} \gau_{[1\;0]} \\
 Y & X\oplus U \gau^{[g\;h]} \dr_{[1\;0]} & X & & & Y\oplus Y &
} \]

Assume $f\sim^l g$. Then there is a commutative diagram:
\[
\xymatrix{
 & X\oplus X \bg_{[f\;g]} \bas^{\left[^{1\;1}_{a\;b}\right]} \bdr^{\nabla} & \\
 Y & X\oplus U \gau^{[u\;v]} \dr_{[1\;0]} & X,
} \]
with $U\in\tc^\perp$. This implies $f-g = v(a-b)$, so that $f-g$ factors through $\tc^\perp$.

The proof that $f\sim^r g$ implies $f-g\in(\tc^\perp)$ is similar.

Finally, if $X$ belongs to $\ccf$, then $f-g$ factors through $\tc^\perp$ if and only if it factors through $\shift\tc$.
\end{proof}

\begin{lemma}\label{lemma: second factorisation}
Any morphism in $\cat$ with cofibrant domain can be factored out as a morphism
in $\,\prescript{\hsq}{}{(I^\square)}\cap\prescript{\bsq}{}{}(I^\square)$ followed by a morphism in $I^\square$.
\end{lemma}

\begin{proof}
Let $f\in\cat(X,Y)$ be a morphism with $X\in\ccf$. Let $T_1 \fl T_0 \fl X \stackrel{\eps}{\fl}\shift T_1$
be a triangle with $T_0,T_1\in\tc$. We note that $\eps$ is a left-$\tc^\perp$ approximation.
Let $QY\stackrel{q_Y}{\gfl}Y$, with $QY\in\ccf$, be given by \cite[Lemma 3.3]{BMloc1} (see Lemma~\ref{lemma: key lemma}).
By Lemma~\ref{lemma: J-inj} and Lemma~\ref{lemma: I-inj}, the morphism $q_Y$ belongs to $I^\square$.
Since $X$ belongs to $\ccf$, there is a lift $\widetilde{f}$ as in the diagram:
\[
\xymatrix{
& QY \bas^{q_Y} \\
X \hdr^{\widetilde{f}}\dr_f & Y.}
\]
Consider the following factorisation of $f$: $X\stackrel{\left[^{\widetilde{f}}_\eps\right]}{\gfl}QY \oplus \shift T_1 \stackrel{[q_Y\;0]}{\gfl}Y$.
The morphism $[q_Y\;0]$ lies in $I^\square$. It remains to show that the morphism $\left[^{\widetilde{f}}_\eps\right]$ satisfies the required
lifting properties. Let thus
\[
\xymatrix{
X \dr^a \bas_{\left[^{\widetilde{f}}_\eps\right]} & A \bas^p \\
QY\oplus \shift T_1 \dr_b & B}
\]
be a commutative square with $p\in I^\square$.
Since $QY\oplus\shift T_1$ belongs to $\ccf$, there is some $QY\oplus\shift T_1\stackrel{g}{\gfl}A$ such that
$pg = b$. Complete $p$ to a triangle $C \stackrel{t}{\fl}A \stackrel{p}{\fl}B\fl \shift C$. Since $pa=pg\left[^{\widetilde{f}}_\eps\right]$,
there is a morphism $X\stackrel{h}{\fl}C$ satisfying $a = g\left[^{\widetilde{f}}_\eps\right] + th$. The morphism $p$ is in
$\wc$ by Lemma~\ref{lemma: I-inj} so that $t$ factors through $\tc^\perp$, and there is some $\shift T_1 \stackrel{k}{\fl}A$ with
$th = k\eps = [0\;k]\left[^{\widetilde{f}}_\eps\right]$. By Lemma~\ref{lemma: htp}, the morphisms $g$ and $g+[0\;k]$
are the two liftings whose existence was claimed in the statement.
\end{proof}

\noindent {\bf Proof of Theorem~\ref{theorem: main}}:
Since $\cat$ is additive, it has finite products and coproducts. Let us check conditions
(1) to (6) of Proposition~\ref{proposition: cof gene}. We note that even when $I$ and $J$ are not sets,
these conditions imply that $\Fib$ and $\Cof$ endow $\cat$ with a left-weak model structure.

(1) Since $\cat$ is a triangulated category, any epimorphism in $\cat$ is isomorphic to some
$Y\oplus Z \stackrel{[1\;0]}{\gfl}Y$. Let $X\stackrel{f}{\fl}Y$ be any morphism in $\cat$.
Then the square
\[
\xymatrix{
X\oplus Z \dr^{[1\;0]} \bas_{\left[^{f\;0}_{0\;1}\right]} & X \bas^f \\
Y\oplus Z \dr^{[1\;0]} & Y
}
\]
is a pull-back square. Moreover, if $f$ is a trivial fibration, then so is $\left[^{f\;0}_{0\;1}\right]$.

(2) The class $W$ is stable under retracts (Lemma~\ref{lemma: retracts})
and satisfies the 2-out-of-3 property (Corollary~\ref{corollary: 2-out-of-3}).

Condition (3) is precisely Corollary~\ref{cor: I-cof cap W}, condition (4) is Lemma~\ref{lemma: I-inj}.
The factorisations (5) and (6) are given respectively by Corollary~\ref{corollary: first factorisation}
and Lemma~\ref{lemma: second factorisation}.

Assume that $\tc$ is skeletally small, and let $t$ be a set of representatives for the
isomrophism classes of objects in $\tc$. Then $\ccf$ is also skeletally small: For each
$T_0,T_1\in t$ and each morphism $T_1\stackrel{\al}{\fl}T_0$ in $\cat$, let
$T_1\stackrel{\al}{\fl}T_0 \fl Y_\al \fl \shift T_1$ be a triangle in $\cat$.
Then the union $y$ of all $Y_\al$'s with $T_0,T_1\in t$ and $\al\in\cat(T_1,T_0)$
form a set of representatives for the isomorphism classes of objects in $\ccf$.
We define $J'$ to be the set of all morphisms $0\fl\shift T$, $T\in t$ and $I'$ to be the union
of all $\cat(T,Y)$ over $T\in t$ and $Y\in y$. Then $I'$ and $J'$ are sets and satisfy
$(I')^\square = I^\square$ and $(J')^\square = J^\square$.\qed

\vspace{7pt}
\noindent {\bf Proof of Corollary~\ref{corollary: main}}:
By Corollary~\ref{corollary: fibrant} and Proposition~\ref{proposition: cofibrant},
the full subcategory of $\cat$ on fibrant and cofibrant objects is $\ccf$.
Moreover, by Lemma~\ref{lemma: htp}, two morphisms $f,g$ in $\ccf$ are homotopic
if and only if $f-g$ factors through $(\shift\tc)$. Corollary~\ref{corollary: main}
thus follows from Theorem~\ref{theorem: main} and Theorem~\ref{theorem: Quillen}.
\qed

\vspace{7pt}
Next proposition is a consequence of Theorem~\ref{theorem: main} and Proposition~\ref{proposition: weq and htpeq}.
We nonetheless include a direct proof.

\begin{prop}\label{proposition: w-eq are htp-eq}
Any weak-equivalence between cofibrant objects
is a homotopy equivalence.
\end{prop}

\begin{proof}
Let $f: X\fl Y$ be a weak equivalence. Complete it to a triangle
$Z \fl X \fl Y\fl \shift Z$. The object $Y$ being cofibrant,
there is (by Proposition~\ref{proposition: cofibrant})
a triangle $T_1 \fl T_0 \fl Y \fl \shift T_1$ in $\cat$,
with $T_0,T_1\in\tc$. Since $f$ is a weak equivalence, the morphisms
$g$ and $h$ in the following diagram belong to $(\tperp)$:
\[
\xymatrix{
& & T_0 \bas^\al & & \\
Z \dr^g & X \dr^f \gex^\eta & Y \gex^\eps \bas_\beta \dr^h & \shift Z \dr^{-\shift g} & \shift X \\
& & \shift T_1 \hdrp_\gamma \upex_\delta & &
}
\]
Since $T_0$ is in $\tc$, the composition $h\al$ vanishes and there is a morphism
$\gamma : \shift T_1 \fl \shift Z$ such that $h = \gamma\beta$. We have moreover
$(\shift g)\gamma = 0$ since $\shift T_1$ belongs to $\shift \tc$ and $\shift g$
to $(\shift\tperp)$. Thus, there is some $\delta : \shift T_1 \fl Y$ with
$\gamma = h\delta$. We have the equalities:
$h\delta\beta = \gamma\beta = h$ so that the morphism $1-\delta\beta$ factors through $f$.
Let $\eps : Y \fl X$ be such that $f\eps + \delta\beta = 1$. Since $\delta\beta$ belongs to $(\shift\tc)$,
the morphsism $\eps$ is a right homotopy inverse of $f$. We claim that $\eps$ is homotopy
inverse to $f$. Indeed, we have the following equalities:
$f \eps f = (1-\delta\beta)f = f - \delta\beta f$, and
$h\delta\beta f = \gamma\beta f = hf = 0$. Thus, there is some
$u : X \fl X$ such that $\delta\beta f = fu$ so that
$f(\eps f - 1 - u) = 0$. There is some $\eta : X \fl Z$ with
$\eps f - 1 = u + g\eta$ and $fu$ belongs to $(\shift\tc)$.
Let us show that $u$ itself belongs to the ideal $(\shift\tc)$. For this, the assumption that
$\tc$ be rigid is needed.
Let $\xymatrix{U_0 \dr^a & X \dr^b & \shift U_1 \dr &}$ be a triangle
in $\cat$ with $U_0,U_1$ in $\tc$. We have $fua = 0$ since
$fu\in (\shift\tc)$ and $\tc$ is rigid. This shows that the morphism
$ua$ factors through $g$, and thus through $\tc^\perp$. Therefore
$ua = 0$ and $u$ factors through $b$ which belongs to $(\shift\tc)$ as
claimed. We have proved that $\eps f -1$ belongs to $(\shift\tc)$
so that $f$ is a homotopy equivalence.
\end{proof}

\subsection{Example}\label{ssection: example1}

In this subsection is given an example of a lifting property $a\hsq g$.
We consider the cluster category $\cat$ of type $\rm{D}_4$ \cite{BMRRT}, over some field $\field$.

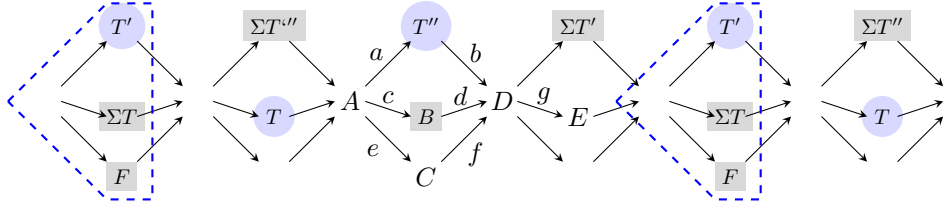
\begin{center}
\begin{tikzpicture}
\draw (2,1) node[circle, fill=blue!15, scale=.8] {$T'$} ;
\draw (4,-.2) node[circle, fill=blue!15, scale=.8] {$T$} ;
\draw (6,1) node[circle, fill=blue!15, scale=.8] {$T''$} ;
\draw (10,1) node[circle, fill=blue!15, scale=.8] {$T'$} ;
\draw (12,-.2) node[circle, fill=blue!15, scale=.8] {$T$} ;
\draw (2,-.2) node[fill=black!15, scale=.8] {$\shift T$} ;
\draw (2,-1) node[fill=black!15, scale=.8] {$F$} ;
\draw (4,1) node[fill=black!15, scale=.8] {$\shift T`''$} ;
\draw (6,-.2) node[fill=black!15, scale=.8] {$B$} ;
\draw (8,1) node[fill=black!15, scale=.8] {$\shift T'$} ;
\draw (10,-.2) node[fill=black!15, scale=.8] {$\shift T$} ;
\draw (10,-1) node[fill=black!15, scale=.8] {$F$} ;
\draw (12,1) node[fill=black!15, scale=.8] {$\shift T''$} ;
\draw (5,0) node {$A$} ;
\draw (6,-1) node {$C$} ;
\draw (7,0) node {$D$} ;
\draw (8,-.2) node {$E$} ;
\draw (5.35,.62) node {$a$} ;
\draw (5.3,-.65) node {$e$} ;
\draw (5.5,.05) node {$c$} ;
\draw (6.65,.65) node {$b$} ;
\draw (6.45,.07) node {$d$} ;
\draw (6.65,-.7) node {$f$} ;
\draw (7.55,.07) node {$g$} ;
\foreach \x in {1,2,...,6} {
 \draw[->] (2*\x-.8,.2) --(2*\x-.2,.8) ;
 \draw[->] (2*\x-.8,0) --(2*\x-.2,-.2) ;
 \draw[->] (2*\x-.8,-.2) --(2*\x-.2,-.8) ;
 \draw[->] (2*\x+.2,.8) --(2*\x+.8,.2) ;
 \draw[->] (2*\x+.2,-.2) --(2*\x+.8,0) ;
 \draw[->] (2*\x+.2,-.8) --(2*\x+.8,-.2) ;
}
\draw[thick, dashed, blue] (.5,0) --(1.8,1.3) --(2.4,1.3) --(2.4,-1.3) --(1.8,-1.3) --cycle ;
\begin{scope}[xshift=8cm] \draw[thick, dashed, blue] (.5,0) --(1.8,1.3) --(2.4,1.3) --(2.4,-1.3) --(1.8,-1.3) --cycle ; \end{scope}
\end{tikzpicture}
\end{center}

The figure above gives the Auslander--Reiten quiver of the cluster category $\cat=\cat_{\rm{D}_4}$.
It should be thought of as lying on a cylinder. In particular, the two parts of this quiver
which have been circled in dashed blue are identified. Possible references on Auslander--Reiten quivers
include \cite{ASS} (for categories of modules) and \cite{Happel} (for derived categories).

The rigid subcategory $\tc$ that we consider is given by $\add(T\oplus T'\oplus T'')$, where the indecomposable objects
$T,T'$ and $T''$ are highlighted in pale blue disks. Indecomposable objects which belong to $\tc^\perp$ are
drawn in light grey boxes.

The morphism $g\in\cat(D,E)$ is part of a(n Auslander--Reiten) triangle:
\[
B \stackrel{d}{\gfl}D\stackrel{g}{\gfl}E\stackrel{\eps}{\gfl}\shift B.
\]
All the triangles in the derived category of $\rm{D}_4$ are described in \cite{KN}, and these
triangles induce triangles in the cluster category $\cat_{\rm{D}_4}$.
As indicated in the figure, $B$ belongs to $\tc^\perp$, so that $\shift B$ (which is $T$)
belongs to $\shift\tc^\perp$. Moreover, the triangle does not split and $\cat(E,T)$ is one dimensional.
It follows that the morphism $\eps$ is given, up to multiplication by a non-zero scalar,
by the path going from $E$ to $T$ via $F$. The latter indecomposable belongs to $\tc^\perp$. This shows that
$g$ is an acyclic fibration. Let $a\in\cat(A,T'')$ be as in the figure.
Let us check that the lifting property $a\hsq g$ holds, but not the lifting property $a\square g$.
For any commutative square:
\[
\xymatrix{
A \dr^x \bas_a & D \bas^g \\
T'' \dr_y & E,
}
\]
we may assume that $y$ equals $gb$ since $\cat(T'',E)$ is one dimensional.
Moreover, $B$ is in $\tc^\perp$, so that $\cat(-,g)|_{\tc}$ is injective and a lift
$\al\in\cat(T'',D)$ with $g\al = y = gb$ is uniquely given by $b$.
The morphism space $\cat(A,D)$ is two-dimensional. By using the relevant mesh relation,
there are two scalars $\ld,\mu\in\field$ such that $x = \ld ba +\mu dc$. Since the square commutes and $gd=0$,
we have $\ld=1$. Therefore, the property $a\hsq g$ holds. Picking $\mu=1$, we have
$x=ba+dc=-fe\neq ba$ so that $a\square g$ does not hold.

\subsection{Some questions under investigation}

In this section, we explain some questions whose answers we would have liked to
give in the present paper, but which are still under investigation.

\subsubsection*{Abelian structure of $\hoc$}

If $\cat$ is a pointed model category, then it is often the case that its homotopy category $\hoc$ is triangulated.
Moreover, the triangulated structure on $\hoc$ is determined by the model structure on $\cat$.
In Theorem~\ref{theorem: main}, the category $\cat$ is triangulated and its homotopy category is abelian
(\cite{KZ,IY,BeligiannisRigid}). We would like to describe this abelian structure directly from the weak model structure on $\cat$.
More generally, it would be interesting to have a sufficient condition on a weak model structure so that the homotopy category is abelian.

\subsubsection*{Calculus of fractions}

In \cite{BMloc2} and \cite{BeligiannisRigid}, it is shown that the localisation of the triangulated category $\cat$
at the class $\wc$ is not far from admitting a calculus of fractions. More precisely, the localisation of $\cat/(\tc^\perp)$
at the image of the class $\wc$ (which turns out to coincide with the class of regular morphisms) admits a calculus of fractions.
Since $\cat/(\tc^\perp)$ is also the category up to homotopy, we expect the existence of the weak model structure on $\cat$ to
give a new proof that $\cat/(\tc^\perp)$ admits a calculus of fractions (by analogy with \cite{Brown}).

\subsubsection*{Hom-infinite cluster categories}

The results in \cite{BMloc1, BeligiannisRigid} as well as Theorem~\ref{theorem: main} do not apply to
Hom-infinite cluster categories. The reason for this is the following: Let $(Q,W)$ be
a Jacobi-infinite quiver with potential and let $\cat_{Q,W}$ be the associated cluster category.
If $\Gamma$ is the image in the cluster category $\cat_{Q,W}$ of the complete Ginzburg dg-algebra 
associated with $(Q,W)$, then $\add\,\Gamma$ is presumably not contravariantly finite.
Our hope is to adapt the proof of Theorem~\ref{theorem: main} to the setup of Hom-infinite cluster categories.

\subsubsection*{Model structures}

Categories with model structures are quite often abelian or exact categories
(see \cite{SchwedeShipley} for many examples and a precise statement).
Here, the category $\cat$ is triangulated so that this is not a natural setup for constructing
model structures. This is the main reason why the structure defined in Theorem~\ref{theorem: main} is
only a weak version of a model structure. It would thus be more natural to adapt the proof of this theorem
to the case of an exact category $\cat$: for example, the Frobenius categories coming from preprojective algebras studied by
Gei\ss--Leclerc--Schr\"oer (see~\cite{GLSsurvey} for a nice survey) and Buan--Iyama--Reiten--Scott~\cite{BIRSc},
or the Frobenius categories constructed by Demonet--Luo in~\cite{DL,DL2}. One might hope
that the techniques from Section~\ref{section: proof} would give rise to model category structures on these Frobenius categories.

\end{document}